\documentclass[12pt,reqno]{amsart}%
\usepackage{amsmath, amsfonts, amssymb, amsthm, amscd, amsbsy}
\usepackage{fancyhdr}
\usepackage[usenames,dvipsnames,svgnames,x11names,hyperref]{xcolor}
\usepackage{geometry}
\usepackage{graphicx}
\usepackage{enumerate}
\usepackage[pagebackref]{hyperref}%
\usepackage{amsmath}%
\usepackage{amsfonts}%
\usepackage{amssymb}

\providecommand{\U}[1]{\protect\rule{.1in}{.1in}}
\hypersetup{
colorlinks=true,
breaklinks=true,
urlcolor=NavyBlue,
linkcolor=Fuchsia,
bookmarksopen=false,
filecolor=black,
citecolor=ForestGreen,
linkbordercolor=red
}
\newtheorem{theorem}{Theorem}[section]

\theoremstyle{definition}
\newtheorem{remark}[theorem]{Remark}
\theoremstyle{plain}

\newtheorem{assumption}{Assumptions}

\newtheorem{corollary}[theorem]{Corollary}

\newtheorem{proposition}[theorem]{Proposition}

\numberwithin{equation}{section}

\newtheorem{theoremalph}{Theorem}

\allowdisplaybreaks

\def\a{\alpha}

\def\e{\epsilon}
\def\la{\lambda}
\def\p{\partial}

\def\grad{\nabla}
\def\D{\nabla}
\def\oo{\infty}
\def\div{\operatorname{div}}
\def\Div{\operatorname{div}}
\def\Ric{\operatorname{Ric}}

\def\Tr{\operatorname{tr}}
\def\R{\mathbb{R}}

\newcommand{\gen}[1]{\ensuremath{\langle #1\rangle}}

\newcommand{\norm}[1]{\left\lVert#1\right\rVert}

\begin{document}
\title[General Conformally Induced Mean Curvature Flow]{General Conformally Induced Mean Curvature Flow}
\author{Joshua Flynn}
\address{Joshua Flynn: CRM/ISM and McGill University\\
Montr\'{e}al, QC H3A0G4, Canada}
\email{joshua.flynn@mcgill.ca}
\author{Jacob Reznikov}
\address{Jacob Reznikov: McGill University\\
Montr\'{e}al, QC H3A0G4, Canada}
\email{yakov.reznikov@mail.mcgill.ca}

\date{\today}

\begin{abstract} 
  This paper continues the investigation of isoperimetric inequalities through volume preserving and area decreasing mean curvature type flows related to conformal Killing vector fields. 
  Results of this kind prior to this paper all studied convex hypersurfaces or hypersurfaces which are starshaped with respect to generalized dilations.
  This paper is the first to study results of this kind for hypersurfaces which are starshaped with respect to general conformal Killing vector fields perturbed by an isometric Killing vector field and our flows allow us to establish isoperimetric inequalities for a much wider class of hypersurfaces.
  For example, our results apply to hypersurfaces in $\R^{n+1}$ which are far from being starshaped in the traditional sense, but are starshaped with respect to the conformal Killing vector field composed of a dilation and a rotational vector field.
  The flow considered in this paper is a novel modification of the mean curvature-type flow first introduced by Guan and Li, which was later generalized by Guan-Li-Wang and Li-Pan.
\end{abstract}

\maketitle
\tableofcontents

\section{Introduction}

The isoperimetric problem asks to find, among all domains of a given volume, those whose boundaries have minimal surface area.
A natural approach to resolve this problem is to consider volume preserving and surface area decreasing geometric flows of the boundary of a domain of a given volume.
In this paper we apply this approach to study the isoperimetric problem for a wide class of domains in general geometries admitting appropriate conformal Killing vector fields.
In the literature, mean curvature type flows have been used to prove isoperimetric inequalities in the category of convex or starshaped domains, where starshapedness has always been relative to a dilation-type conformal Killing vector field.
Our paper pushes this forward to a wider category of domains where the usual starshapedness is replaced by starshapedness relative to more general conformal Killing vector fields composed of a dilation-type part and a rotational-type part.
Under appropriate geometric assumptions, this allows one to establish the isoperimetric inequality by geometric flows in all dimensions $\geq 3$ for a wider category of domains and this is even new for Euclidean space.
To achieve this, we introduce a novel modification to a type of mean curvature flow first introduced by Guan-Li, which will now be described.

Recently this flow approach was used by Guan-Li in \cite{guan_mean_2013} where, based on the Minkowski identities, they introduced a volume preserving and area decreasing mean curvature type flow to prove the isoperimetric inequality for strictly starshaped closed hypersurfaces in space forms.
To state the flow on $\R^{n+1}$, let $\Sigma \subset \R^{n+1}$ be a smooth closed hypersurface enclosing the origin, let $H$ be the mean curvature, let $\nu$ be the outward unit normal, let $X$ be the position vector field and let $u = \gen{X,\nu}$ be the support function.
Then the Guan-Li flow is
\begin{equation}
  \frac{\p X}{\p t} = (n - Hu) \nu.
  \label{eq:guan-li-flow}
\end{equation}
Moreover, using the evolution equations for the volume enclosed $V=V(\Sigma)$ and surface area $A=A(\Sigma)$ of $\Sigma$, the Minkowski identities 
\begin{equation}
  (k+1) \int_{\Sigma} \sigma_{k+1}(\kappa) u = (n-k) \int_{\Sigma} \sigma_{k}(\kappa),\quad \sigma_{k}(\kappa) = k\text{ -curvature},
  \label{eq:general-minkowski-identities}
\end{equation}
give $\p_{t} V=0$ and $\p_{t} A \leq 0$.
A strictly starshaped domain is one for which $u >0$ and Guan-Li showed that the solution hypersurfaces to \eqref{eq:guan-li-flow} with strictly starshaped initial data $\Sigma_{0}$ exponentially converge to a sphere $S$ with $V(\Sigma_{0}) = V(S)$.

In \cite{guan_volume_2018}, Guan-Li-Wang significantly extended these results to a large class of warped product spaces.
More precisely, let $(N^{n+1},\overline{g})$ be a warped product space with closed base $(B^{n},\tilde{g})$ and warped metric 
\[
  \overline{g} = dr^{2} + \phi^{2} \tilde{g}, \quad r \in [r_{0},r_{1}],
\]
where $\phi = \phi(r)$ is a smooth positive function defined on $[r_{0},r_{1}]$.
On $N^{n+1}$ is the conformal Killing vector field $X = \phi(r)\p_{r}$ of dilation type.
Note that, on $\R^{n+1}$, the position vector field $X$ may be expressed as $X=x \cdot \grad = r \p_{r}$, and is a conformal Killing vector field corresponding to Euclidean dilation.
The Guan-Li-Wang flow on warped product spaces is then given by
\begin{equation}
  \frac{\p F}{\p t} = (n\phi' - Hu) \nu,
  \label{eq:guan-li-wang}
\end{equation}
where $F$ is a family of embeddings.
That \eqref{eq:guan-li-wang} is volume preserving and area decreasing follows from the conformality of $X$ and Minkowski identities on $N$, also shown in \cite{guan_volume_2018}.
In \cite{guan_volume_2018} they considered initial data $\Sigma_{0}$ which are smooth graphical hypersurfaces in $N$, namely, $\Sigma$ is defined via $r = \rho (p)$ for $p \in B$, where $\rho$ is some smooth function on $B$.
In particular, $u>0$ along $\Sigma_{0}$ and hence $\Sigma_{0}$ is starshaped.
To obtain existence of a smooth solution for infinite time and exponential convergence to a level set of $r$ as $t \to \oo$, they impose the ambient conditions $\Ric_{\tilde{g}} \geq (n-1) K \tilde{g}$ and $0 \leq (\phi')^{2} - \phi''\phi \leq K$ on $[r_{0},r_{1}]$.
(See also \cite{MR4429248} for an exploration on the latter condition.)
Their result is stated as follows.

\begin{theoremalph}[Guan-Li-Wang]
  Let $\Sigma_{0}$ be a smooth graphical hypersurface in $(N^{n+1},\overline{g})$ with $n \geq 2$.
  If $\phi(r)$ and $\tilde{g}$ satisfy the conditions
  \begin{align*}
    &(n-1)K \tilde{g}  \leq \Ric_{\tilde{g}}\\
    &0  \leq (\phi')^{2} - \phi''\phi \leq K \, \text{ on } \, [r_{0},r_{1}],
  \end{align*}
  where $K>0$ is a constant, then the evolution equation \eqref{eq:guan-li-wang} with $\Sigma_{0}$ as initial data has a smooth solution for $t \in \R_{\geq0}$.
  Moreover, the solution hypersurfaces converge exponentially to a level set of $r$ as $t \to \oo$.
  \label{intro-thm:guan-li-wang}
\end{theoremalph}

In \cite{guan_volume_2018}, they also proved an isoperimetric inequality, which is now described.
Let $S(r)$ be a level set and $B(r)$ the region bounded by $S(r)$ and $S(r_{0})$.
Define the function $\xi:\R_{\geq0} \to \R_{\geq0}$ as the one satisfying $A(r) = \xi(V(r))$, where $A(r)$ is the area of $S(r)$ and $V(r)$ the volume of $B(r)$, noting that such a function is always well-defined.
The isoperimetric inequality proved by Guan-Li-Wang is in terms of $\xi$ and is stated in the following theorem.

\begin{theoremalph}[Guan-Li-Wang]
  Assume $N,\phi,\tilde{g}$ satisfy the hypotheses in Theorem \ref{intro-thm:guan-li-wang}.
  Let $\Omega \subset N$ be a domain bounded by a smooth graphical hypersurface $M$ and $S(r_{0})$.
  Then
  \begin{equation}
    \xi(Vol(\Omega)) \leq Area(M),
    \label{eq:guan-li-wang-isoperimetric-inequality}
  \end{equation}
  where $Vol(\Omega)$ is the volume of $\Omega$, $Area(M)$ the area of $M$ and $\xi$ is defined above.
  If we additionally assume either $(\phi')^{2} - \phi''\phi<K$ or $(n-1)K \tilde{g} < \Ric_{\tilde{g}}$ on $[r_{0},r_{1}]$, then equality is achieved in \eqref{eq:guan-li-wang-isoperimetric-inequality} if and only if $M$ is a level set of $r$.
  \label{intro-thm:guan-li-wang-isoperimetric}
\end{theoremalph}

\medskip

Very recently, the results of  Guan-Li-Wang in \cite{guan_volume_2018} were extended by Jiayu Li and Pan in \cite{jiayu_isoperimetric_2023} to more general manifolds admitting a dilation-type conformal Killing vector field.
To state their results, let $(N^{n+1},\overline{g})$ be a general Riemannian manifold endowed with a complete conformal Killing vector field $X$, let $\varphi = \frac{\div X}{n+1}$ and let $\mathcal{F}(X)$ be the foliation in $N' = N \setminus \left\{ X=0 \right\}$ determined by the $n$-dimensional distribution $\mathcal{D} = \left\{ Y \in TN: \overline{g}(X,Y) = 0 \right\}$.
Lastly, call a closed hypersurface $\Sigma \subset N'$ strictly starshaped relative to $X$ provided $\gen{X,\nu} > 0$ along $\Sigma$.
Li-Pan considered the flow
\begin{equation}
  \frac{\p F}{\p t} = (n \varphi - u H)\nu,
  \label{eq:li-pan-flow}
\end{equation}
where $u = \gen{X,\nu}$ and the initial data $\Sigma_{0} \subset N'$ is assumed to be a closed embedded hypersurface that is strictly starshaped.
Note that \eqref{eq:li-pan-flow} reduces to \eqref{eq:guan-li-wang} in the warped product setting when $X$ is chosen to be $\phi(r)\p_{r}$.
To state one of their main results, we introduce the following set of assumptions.

\begin{assumption}
  Let $(N^{n+1},\overline{g})$ be a Riemannian manifold, let $X \in TN$ and let $N' = N \setminus \left\{ X=0 \right\}$.
  \begin{enumerate}[(i)]
    \item $X$ is a conformal Killing vector field, i.e., $\mathcal{L}_{X} \overline{g} = 2 \varphi \overline{g}$;
    \item $\varphi = \frac{\div_{\overline{g}}X}{n+1} > 0$ on $N'$;
    \item $\varphi^{2} - X(\varphi)>0$ on $N'$;
    \item defining the distribution $\mathcal{D} = \left\{  Y \in TN: \overline{g}(X,Y) = 0\right\}$, let $\mathcal{F}(X)$ be the corresponding foliation in $N'$ and assume each connected leaf of $\mathcal{F}(X)$ is a constant mean curvature closed hypersurface and a level set of $|X|/\varphi$;
    \item for all $Y \in TN'$, there holds
      \[
	|X|^{2} \Ric_{\overline{g}}(Y,Y) - |Y|^{2} \Ric_{\overline{g}}(X,X) \geq 0,
      \]
      i.e., the direction $X/|X|$ is of least Ricci curvature in $N'$.
  \end{enumerate}
  \label{assumption1}
\end{assumption}

\begin{theoremalph}[Li-Pan]
  Let $N$ and $X$ satisfy Assumptions \ref{assumption1} and let $\Sigma_{0} \subset N'$ be a closed embedded hypersurface strictly starshaped relative to $X$.
  Then the flow \eqref{eq:li-pan-flow} with initial data $\Sigma_{0}$ has a smooth solution for infinite time and which converges to a totally umbilical hypersurface whose unit normal field $\nu_{\oo}$ attains least Ricci curvature on $N'$, namely,
  \[
    \Ric_{\overline{g}}(\nu_{\oo},\nu_{\oo}) = \Ric_{\overline{g}}(\mathcal{N},\mathcal{N}),
  \]
  where $\mathcal{N} = X/|X|$.
  \label{intro-thm:li-pan}
\end{theoremalph}

We note that Assumptions \ref{assumption1}(iv) implies that each leaf of $\mathcal{F}(X)$ is totally umbilical with mean curvature $n \varphi/|X|$ (see \cite[Prop. 2.3]{jiayu_isoperimetric_2023}).
Moreover, Assumptions \ref{assumption1}(iii) is equivalent to the (strengthened) Guan-Li-Wang assumption $0 < (\phi')^{2} - \phi''\phi$ when $N^{n+1}$ has a warped product structure.

\medskip 

Our main result is that, by introducing a novel modification of the mean curvature-type flows considered in \cite{guan_mean_2013,guan_volume_2018,jiayu_isoperimetric_2023} and using certain techniques  from \cite{jiayu_isoperimetric_2023}, the conclusions of Theorem \ref{intro-thm:li-pan} hold for more general initial data $\Sigma_{0}$ which are starshaped with respect to more general conformal Killing vector fields.
Our approach is to replace the conformal Killing vector field considered by Li-Pan with a time-dependent conformal Killing vector field which splits into two parts: $X(t) = X^{\top}(t) + X^{\perp}$, where $X^{\top}(t)$ is a time-dependent Killing vector field vanishing in finite time $T_{0}$ and $X^{\perp}$ is a conformal Killing vector field of the same kind considered by Li-Pan. 
We show that, if $\Sigma_{0}$ is strictly starshaped relative to $X(0)$, then we can flow it to a surface which is starshaped relative to $X(T_{0}) = X^{\perp}$ and whence apply Theorem \ref{intro-thm:li-pan}.
We note that $T_{0}=T_{0}(\Sigma_{0})$ and we need only additionally assume $X^{\top}$ generates an integrable distribution in a sufficiently large neighborhood of $\Sigma_{0}$.
Informally, we may consider $X^{\perp}$ as a dilation-type vector field and $X^{\top}$ as a rotational-type vector field; comparing with the Euclidean space makes this interpretation obvious.
We emphasize that our paper is inspired by the work of Li-Pan in \cite{jiayu_isoperimetric_2023} and that, while many calculations reduce to that of \cite{jiayu_isoperimetric_2023} when $X^\top = 0$, the case $X^\top$ is nonzero introduces significant and interesting obstacles that were nontrivial to overcome in this paper.

Our results also improve those of Guan-Li \cite{guan_mean_2013} and Guan-Li-Wang \cite{guan_volume_2018}, and in particular the Euclidean setting.
Recall that their results applied to $\R^{n+1}$ only hold for initial data starshaped relative to the dilation vector field $\rho \p_{\rho}$, i.e., starshaped domains in the traditional sense.
Our results apply to the case $\rho \p_{\rho}$ is perturbed by an isometric Killing vector field arising from rotation; e.g., we may take 
\begin{align*}
  X(0) &= \sum_{j=1}^{n+1} x_{j} \p_{x_{j}} + x_{1} \p_{x_{2}} - x_{2} \p_{x_{1}},
\end{align*}
where
\begin{align*}
  X^{\perp} &= \sum_{j=1}^{n+1} x_{j} \p_{x_{j}} = \rho \p_{\rho}, \qquad X^{\top} = x_{1}\p_{x_{2}} - x_{2} \p_{x_{1}}.
\end{align*}
It is important to note that hypersurfaces starshaped relative to $X(0)$ of this kind may be very far from being starshaped in the usual sense (see Figures \ref{spiral1} and \ref{spiral2}) and therefore our paper provides a unified proof for a larger class of domains.
This is substantial because a unified proof of the isoperimetric inequality via geometric flows in $\R^{n+1}$ for all $n \geq 3$ has yet to be given.

\begin{figure}[!htb]
   \begin{minipage}{0.48\textwidth}
     \centering
     \includegraphics[width=.7\linewidth]{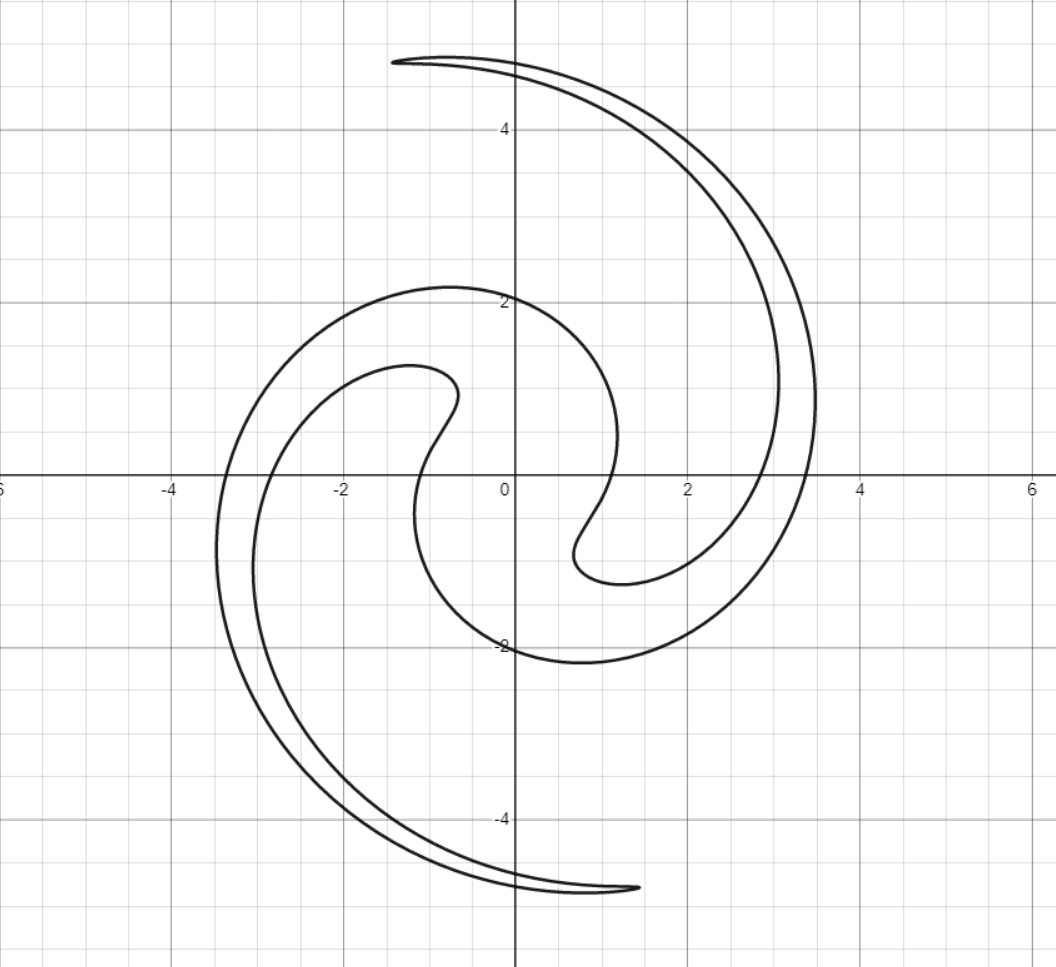}
     \caption{}
     \label{spiral1}
   \end{minipage}\hfill
   \begin{minipage}{0.48\textwidth}
     \centering
     \includegraphics[width=.7\linewidth]{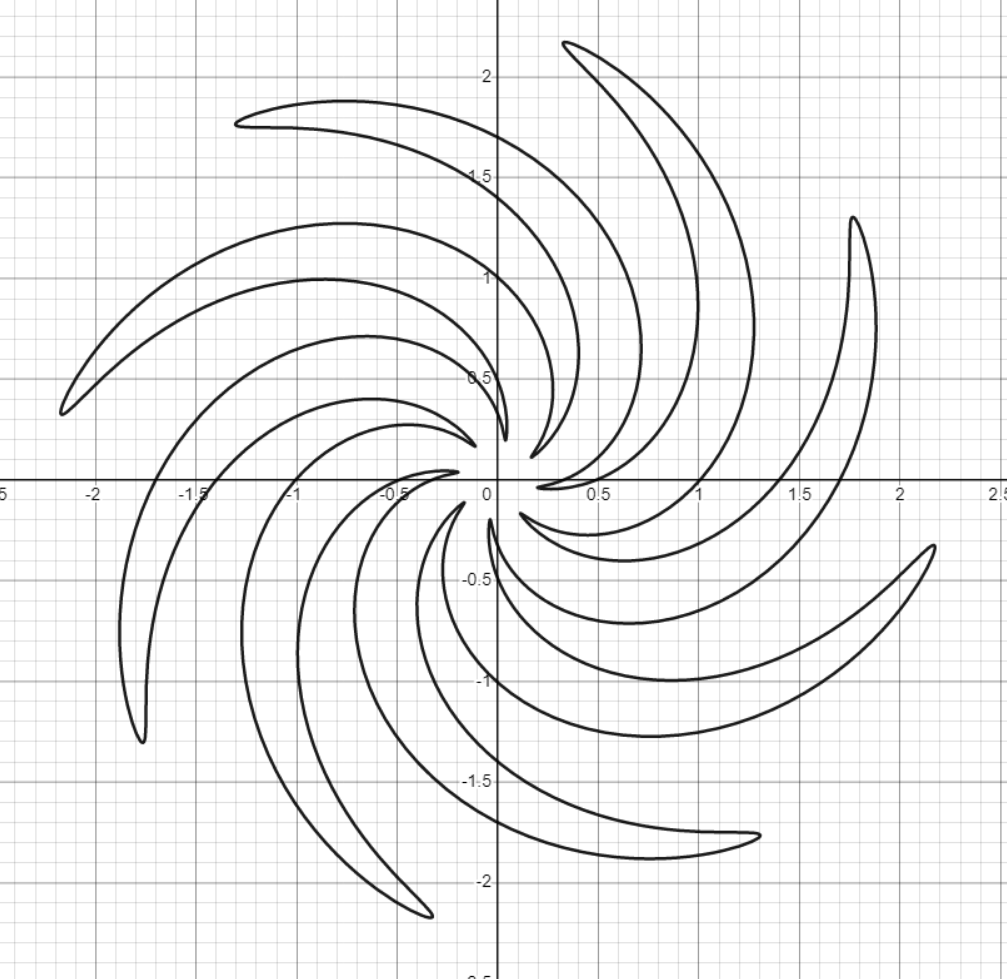}
     \caption{}
     \label{spiral2}
   \end{minipage}
\end{figure}

\medskip

To precisely state our main result, we fix some notation.
Let $(N^{n+1},\overline{g})$, $n\geq2$, be a Riemannian manifold, let $U \subset N$ be an open subset admitting a conformal Killing vector field $X \in TU$, let $u = \gen{X,\nu}$ be the support function for a given hypersurface in $U$ with outward unit normal $\nu$, let $\mathfrak{F} = \left\{ S_{\a} \right\}_{\a}$ be a codimension 1 foliation of $U$ by compact connected leaves $S_{\a}$, let $\mathcal{D} \subset TU$ be the induced distribution, let $P_{D}$ denote the orthogonal projection sending $x \in T_{p}U$ to $P_{D}x \in \mathcal{D}_{p}$ for $p \in U$ and let $\mathcal{D}^{\perp}$ denote the orthogonal complement bundle of $\mathcal{D}$.
For $Y \in TU$, write $Y^{\top} = P_{D}Y$ and whence the decomposition $Y = Y^{\top} + Y^{\perp}$, $Y^{\perp} \in \mathcal{D}^{\perp}$.
Relative to the leaves of $\mathfrak{F}$, we understand $X^{\top}$ and $X^{\perp}$ as the tangential and perpendicular parts of $X$.
Lastly, we introduce the following symmetry assumptions on $(N,U,\mathfrak{F},X)$.
\begin{assumption}
  Let $(N^{n+1},\overline{g})$ be a Riemannian manifold, $U \subset N$ an open subset, $\mathfrak{F}$ a foliation of $U$ and $X \in TU$.
  \noindent 
  \begin{enumerate}[(i)]
    \item $X$ is a conformal Killing vector field on $U$, i.e., $\mathcal{L}_{X} \overline{g} = 2 \varphi \overline{g}$;
    \item $\varphi= \frac{\div_{\overline{g}} X}{n+1} > 0$ and $X^{\perp} \neq 0$ in $U$;
    \item each connected leaf of $\mathfrak{F}$ is strictly starshaped with respect to $X$, i.e., $u > 0$ along each $S_{\a}$;
    \item the mapping $\lambda : U \to \R$, $p \mapsto |X^{\perp}|^{2}/\varphi^{2}$ is constant along each connected leaf of $\mathfrak{F}$;
    \item using (iv) to write $\overline\grad \lambda = 2\Lambda X^{\perp}$ for some smooth function $\Lambda$, then $\Lambda>0$ on $U$;
    \item additionally we assume that $\Lambda \varphi^3 - X^\top(\varphi) > 0$ on $U$;
    \item $X^{\top}$ is a Killing vector field on $U$;
    \item the distribution $\mathcal{E} = \left\{ Y \in TU : (Y,X^{\top})=0 \right\}$ is integrable on $U \setminus \left\{ X^{\top}=0 \right\}$, equivilantly $(X^\top)^\sharp \wedge \mathrm{d} ((X^\top)^\sharp) = 0$ where $(X^\top)^\sharp$ is the dual one-form given by the metric.
    \item for all $Y \in TU$, there holds
      \[
	|X^{\perp}|^{2} \Ric_{\overline{g}}(Y,Y) - |Y|^{2} \Ric_{\overline{g}}(X^{\perp},X^{\perp}) \geq 0,
      \]
      i.e., the direction $\mathcal{N}^{\perp} := X^{\perp}/|X^{\perp}|$ is of least Ricci curvature in $U$.
    \item The same condition holds for $X^\top$, in particular wherever $X^\top \neq 0$ we have 
    \[
      \Ric_{\overline{g}}(\mathcal{N}^\top, \mathcal{N}^\top) = \Ric_{\overline{g}}(\mathcal{N}^\perp, \mathcal{N}^\perp),
    \]
    where $\mathcal{N}^{\top} = X^{\top} / |X^{\top}|$.
  \end{enumerate}
  \label{assumption2}
\end{assumption}

We remark that Assumptions \ref{assumption2} is a formal description of a space $N$ which admits (on $U$) a dilation-type vector field $X^{\perp}$ and compatible rotational-type vector field $X^{\top}$ whose integral curve flow fixes each leaf $S_{\a} \in \mathfrak{F}$.
As such, spaces which satisfy these assumptions naturally generalize the setting of space forms since these naturally have vector fields which satisfy Assumptions \ref{assumption2}.

We now describe the flow we study in this paper.
Let $X$ satisfy Assumptions \ref{assumption2} and let $\Xi:\R_{\geq0} \to \R_{\geq0}$ be a decreasing smooth function satisfying $\Xi(0) = 1$ and $\Xi(t)=0$ for $t \geq 1$.
Define the time-dependent vector field $X(t) = X^{\perp} + \Xi(t) X^{\top}$ and corresponding support function $u(t) = \gen{\nu,X(t)}$.
It is clear $X(t)$ satisfies Assumptions \ref{assumption2} for $t<1$ whenever $X(0)=X$ does.
The flow we consider is 
\[
  \frac{\p F}{\p t} = (n \varphi - u(t) H)\nu = (n \varphi - \gen{X^{\perp} + \Xi(t) X^{\top},\nu}H)\nu.
\]
Note that, for $t \geq 1$, this flow agrees with the flow \eqref{eq:li-pan-flow} considered by Li-Pan.
Our main result may now be stated.

\begin{theorem}
  Let $(N,U,\mathfrak{F},X)$ satisfy Assumptions \ref{assumption2} and let $\Sigma_{0} \subset U$ be a closed embedded hypersurface strictly starshaped relative to $X=X(0)$.
  Then there exists a $T_{0}=T_{0}(\Sigma_{0})>0$ such that the flow
  \begin{equation}
    \frac{\p F}{\p t} = (n \varphi - u(t/T_{0}) H)\nu
    \label{eq:main-flow}
  \end{equation}
  with initial data $\Sigma_{0}$ has a smooth solution for infinite time that converges to a totally umbilical hypersurface $\Sigma_{\oo}$ whose unit normal field $\nu_{\oo}$ attains least Ricci curvature on $U$, namely,
  \[
    \Ric_{\overline{g}}(\nu_{\oo},\nu_{\oo}) = \Ric_{\overline{g}}(\mathcal{N}^{\perp},\mathcal{N}^{\perp}),
  \]
  where $\mathcal{N}^{\perp} = X^{\perp}/|X^{\perp}|$.
  \label{thm:main}
\end{theorem}

Strengthening Assumptions \ref{assumption1}(v) with 
\begin{enumerate}
  \item[\textit{(v)'}] \textit{The direction $\mathcal{N}^{\perp} = X^{\perp}/|X^{\perp}|$ is the only of least Ricci curvature $N'$}
\end{enumerate}
Li-Pan concluded that the limit surface of their flow is a leaf in $\mathcal{F}(X)$ and that the isoperimetric inequality holds.
We show that this assumption is unnecessary in Section \ref{sec:convergence} and so we obtain the same convergence result and an improvement on the isoperimetric inequalities in \cite{guan_mean_2013,guan_volume_2018,jiayu_isoperimetric_2023}.
In preparation, let $S(r) = \left\{ \la^{1/2} := |X^{\perp}|/\varphi = r  \right\}$ be a level set of $\la^{1/2}$ and let $B(r,r_{0})$ denote the domain bounded between $S(r)$ and $S(r_{0})$ for $r > r_{0} \geq 0$.
Then the following corollary is a direct consequence of Theorem \ref{thm:main} and Section \ref{sec:convergence}.
\begin{corollary}
  Supposing Assumptions \ref{assumption2} hold and $\Sigma_{0} \subset U$ is a strictly starshaped domain relative to $X$, then the limit surface $\Sigma_{\oo}$ from Theorem \ref{thm:main} is a leaf in the foliation $\mathfrak{F}$.
  \label{cor:convergence-result}
\end{corollary}

Using Corollary \ref{cor:convergence-result}, the volume preservation and area decreasing properties of our flow, we obtain the following isoperimetric inequality.

\begin{theorem}
  Assume Assumptions \ref{assumption2}, let $r_{0} \geq 0$, let $\Sigma_{0} \subset U$ be a closed starshaped hypersurface properly enclosing $S(r_{0})$, let $\Omega$ be the domain bounded by $\Sigma_{0}$ and $S(r_{0})$ and let $r_{1}>r_{0}$ be the unique number such that $V(B(r_{1},r_{0})) = V(\Omega)$.
  Then
  \[
    A(S(r_{1})) \leq A(\Sigma_{0})
  \]
  and equality holds if and only if $\Sigma_{0} = S(r_{1})$.
  \label{thm:main-isoperimetric-inequality}
\end{theorem}

We mention that those results in \cite{jiayu_isoperimetric_2023} concerning replacing (v)' with an estimate on the sectional curvature or assuming the conformal Killing vector field is closed may also be extended in the same way our Theorem \ref{thm:main} extends their Theorem \ref{intro-thm:li-pan}.
We leave it to the reader to make these conversions.

\medskip 
We lastly mention that using flows to prove isoperimetric inequalities is not new.
Indeed, Huisken studied in \cite{huisken_flow_1984} a volume preserving and area decreasing flow given by the normalized MCF $X_{t} = (n^{-1}c(t) - H)\nu$, where $c(t) = \int_{M}H d\mu/\int_{M} d\mu$ and $d\mu$ is the surface volume form.
While this flow is suitable to study the isoperimetric problem for convex surfaces, it is a nonlocal flow and existence of solutions for nonconvex hypersurfaces is unclear.
The advantage of the Guan-Li flow is that it is suitable for studying flows of initial data without curvature assumptions and $C^{0},C^{1}$-estimates can be directly obtained to show long-time existence and convergence.
We also mention \cite{schulze_nonlinear_2006} where Schulze used a mean curvature type flow defined in terms of powers of the mean curvature to prove the isoperimetric inequality in $\R^{n+1}$ for $n \leq 7$ and \cite{gage_heat_1986} where Gage-Hamilton proved the isoperimetric inequality for convex planar domains via the curve shortening flow. 

The outline of the paper is as follows.
We first begin with geometric preliminaries to help with later calculations (see Section \ref{sec:preliminaries}). 
We then compute the evolution equations under the flow for the a specified scale function to prove compactness for all time and then use that to prove existence for all time (see Section \ref{sec:evolution-equations}).
Next, we argue long time existence by appropriately changing the flow depending on the initial surface (see Section \ref{sec:main-proof}).
We then give a novel proof that the surface converges to a leaf of the foliation adding no extra conditions (see Section \ref{sec:convergence}); explicit details concerning convergence are relegated to the Appendix.
Lastly, we provide an explicit nontrivial example in Section \ref{sec:examples} for which our results apply.

\section{Preliminaries}\label{sec:preliminaries}

We henceforth use the same notation as that given in the paragraph preceding Assumptions \ref{assumption2} and we always assume Assumptions \ref{assumption2} holds.

We begin by setting additional notation.
Let $\mathfrak{T}$ be the foliation determined by the distribution $\mathcal{E} = \left\{ Y \in TU : (Y,X^{\top}) = 0 \right\}$.
Let $u^{\top} = (X^{\top},\nu)$ and $u^{\perp} = (X^{\perp},\nu)$.
For brevity, a conformal Killing vector field and a Killing vector field will, respectively, be referred to as being conformal and isometric.
Ambient geometric objects and quantities are indicated with a bar; e.g., $\overline{\Ric}$ indicates the Ricci curvature of $\overline{g}$.
Otherwise, $g$ generally denotes the induced metric on a given hypersurface and we often write $(X,Y)$ for $\overline{g}(X,Y)$.
Lastly, we record
\begin{equation}
  \begin{aligned}
  \overline{g}(\D_{Y} X, Z) + 
  \overline{g}(\D_{Z} X, Y)
  &=
  2 \varphi \overline{g}(Y,Z)\\
  \overline{g}(\D_{Y} X^\perp, Z) + 
  \overline{g}(\D_{Z} X^\perp, Y)
  &=
  2 \varphi \overline{g}(Y,Z)\\
  \overline{g}(\D_{Y} X^\top, Z) + 
  \overline{g}(\D_{Z} X^\top, Y)
  &=
  0,
  \end{aligned}
  \label{eq:conformal-vector-field-metric-equation}
\end{equation}
which follows from the conformality of $X$ and $X^{\perp}$ and that $X^{\top}$ is isometric.

Now, we first show that the connected leaves of $\mathfrak{F}$ and $\mathfrak{T}$ are umbilic and compute their mean curvatures.
\begin{proposition}
  The connected leaves of $\mathfrak{F}$ and $\mathfrak{T}$ are umbilic and, respectively, have mean curvature $n \lambda^{-1/2} = n\varphi/|X^{\perp}|$ and $0$.
  \label{prop:umbilic-mean-curvature}
\end{proposition}

\begin{proof}
  Let $S_{\a}$ be a connected leaf of $\mathfrak{F}$ and fix an orthonormal frame $\left\{ e_i \right\}$ for $S_\alpha$.
  The symmetry of the second fundamental form gives 
    \begin{align*}
        h_{ij} &= g(\grad_{e_i} \nu, e_j)
        = g(\grad_{e_i} \frac{X^\perp}{|X^\perp|}, e_j)= \frac{1}{|X^\perp|}g(\grad_{e_i} X^\perp, e_j) 
        \\ &= \frac{1}{|X^\perp|} g(\grad_{e_i} X^\perp, e_j)
        = \frac{1}{|X^\perp|} g(\grad_{e_i} (X - X^\top), e_j) 
        \\ &=
        \frac{1}{|X^\perp|} \left( g(\grad_{e_i} X, e_j) -
         g(\nabla_{e_i} X^\top, e_j)  \right).
    \end{align*}
    Symmetrizing this expression for $h_{ij}$, using that $X$ is conformal and that $X^{\top}$ is isometric along each $S_{\a}$, we get from \eqref{eq:conformal-vector-field-metric-equation} that
    \begin{align*}
        h_{ij} &= \frac{1}{|X^\perp|} \varphi \overline{g}_{ij}  = \lambda ^{-\frac{1}{2}} \overline{g}_{ij},
    \end{align*}
    and so $S_\alpha$ is totally umbilical with mean curvature $n \lambda ^{-1/2}$.
    
    Next, fix an orthonormal frame $\left\{ e_i \right\}$ for a connected leaf $T$ in $\mathfrak{T}$.
    The symmetry of the second fundamental form gives 
    \begin{align*}
        h_{ij} &= g(\grad_{e_i} \nu, e_j)
        = g(\grad_{e_i} \frac{X^\top}{|X^\top|}, e_j)= \frac{1}{|X^\top|}g(\grad_{e_i} X^\top, e_j) 
        \\ &= \frac{1}{|X^\top|} g(\grad_{e_i} X^\top, e_j).
    \end{align*}
    Symmetrizing this expression for $h_{ij}$ and again using \eqref{eq:conformal-vector-field-metric-equation} we immediately get $T$ is totally umbilic with zero mean curvature.
\end{proof}

\begin{remark}\label{rem:symmetry}
  Following the same reasoning as the proof of Proposition \ref{prop:umbilic-mean-curvature}, we have for all $Y,Z \in \mathcal{D}$ or $Y,Z \in \mathcal{D}^{\perp}$:
    \[
        (\grad_Y X^\perp, Z) = \varphi (Y,Z).
    \]
\end{remark}

Recall Assumptions \ref{assumption2}(iv) gives $\grad \la = 2\Lambda X^{\top}$ for some smooth function $\Lambda$.
The following proposition gives $\Lambda$ explicitly.

\begin{proposition}
  There holds
    \[
        \Lambda = \frac{\varphi^2 - X^\perp(\varphi)}{\varphi^3} 
    \]
\end{proposition}
\begin{proof}
    We have
    \[
        X^\perp (\lambda ) = (X^\perp, \grad \lambda ) = 2 \Lambda \cdot |X^\perp|^2
    \]
    and we can compute 
    \begin{align*}
        X^\perp \left( \frac{|X^\perp|^2}{\varphi^2} \right)
        &=
        -\frac{|X^\perp|^2}{\varphi^4} X^\perp(\varphi^2)
        + \frac{1}{\varphi^2} X^\perp (X^\perp, X^\perp) \\
        &=
        -2\frac{|X^\perp|^2}{\varphi^3} X^\perp(\varphi)
        + 2 \frac{1}{\varphi^2} (\grad_{X^\perp} X^\perp, X^\perp).
    \end{align*}
    By the symmetry of the last expression and Remark \ref{rem:symmetry}, we have
    \[
        2 \Lambda \cdot |X^\perp|^2 = 2\frac{\varphi^2 |X^\perp|^2 - |X^\perp|^2 X^\perp(\varphi)}{\varphi^3},
    \]
    giving us the desired result.
\end{proof}

\begin{proposition}
    \label{prop:grad-of-xtop-xperp}
    For $Y \in TU$, we have 
    \[
        \grad_Y X^\perp
        =
        \varphi Y +  \frac{Y(\varphi)}{\varphi} X^\perp  - \frac{(X^\perp,Y)}{\varphi} \grad(\varphi)
    \]
    and
    \[
        \grad_Y X^\top
        =
        \frac{Y(|X^\top|)}{|X^\top|} X^\top  - \frac{(X^\top,Y)}{|X^\top|} \grad(|X^\top|)
    \]
    as long as $|X^\top| \neq 0$.
\end{proposition}

\begin{proof}
  Take some $Y,Z \in TU$ and we compute
  \begin{align*}
    (\grad_Y X^\perp,Z)
    &= \left( \grad_{Y} X^\perp, Z^\perp \right)
    + \left( \grad_{Y} X^\perp, Z^\top \right).
  \end{align*}
  Now for the first term we have 
  \begin{align*}
    \left( \grad_{Y} X^\perp, Z^\perp \right)
    &= \frac{(X^\perp,Z^\perp)}{(X^\perp, X^\perp)}\left( \grad_{Y} X^\perp, X^\perp \right)
    \\ &= \frac{(X^\perp,Z^\perp)}{2(X^\perp, X^\perp)}Y(|X^\perp|^2).
  \end{align*}
  For the second term we have 
  \begin{align*}
    \left( \grad_{Y} X^\perp, Z^\top \right)
    &=
    \left( \grad_{Y^\top} X^\perp, Z^\top \right)
    +
    \left( \grad_{Y^\perp} X^\perp, Z^\top \right)
    \\ &=
    \varphi (Y^\top, Z^\top) 
    +
    \left( \grad_{Y^\perp} X^\perp, Z^\top \right)
    \\ &=
    \varphi (Y^\top, Z^\top) 
    -
    \left( \grad_{Z^\top} X^\perp, Y^\perp \right).
  \end{align*}
  Using that $Y^{\perp}$ is colinear with $X^{\perp}$, we get
  \begin{align*}
    \left( \grad_{Y} X^\perp, Z^\top \right)
    &=
    \varphi (Y^\top, Z^\top) 
    -
    \frac{(Y^\perp,X^\perp)}{(X^\perp,X^\perp)}\left( \grad_{Z^\top} X^\perp, X^\perp \right)
    \\ &=
    \varphi (Y^\top, Z^\top)  
    -
    \frac{(Y^\perp,X^\perp)}{2(X^\perp,X^\perp)} Z^\top(|X^\perp|^2)
    \\ &=
    \varphi (Y^\top, Z^\top)  
    -
    \frac{(Y^\perp,X^\perp)}{2(X^\perp,X^\perp)} (Z(|X^\perp|^2) - Z^\perp(|X^\perp|^2))
    \\ &=
    \varphi (Y^\top, Z^\top)  
    -
    \frac{(Y^\perp,X^\perp)}{2(X^\perp,X^\perp)}Z(|X^\perp|^2) + \frac{(Y^\perp,X^\perp)(Z^\perp,X^\perp)}{2(X^\perp,X^\perp)} 2\varphi(X^\perp, X^\perp)
    \\ &=
    \varphi (Y^\top, Z^\top)  
    -
    \frac{(Y^\perp,X^\perp)}{2(X^\perp,X^\perp)}Z(|X^\perp|^2) + \varphi (Y^\perp,Z^\perp)
    \\ &=
    \varphi (Y, Z)  
    -
    \frac{(Y^\perp,X^\perp)}{2(X^\perp,X^\perp)}Z(|X^\perp|^2) .
  \end{align*}
  Combining the expressions for $(\grad_{Y} X^{\perp}, Z^{\perp})$ and $(\grad_{Y} X^{\perp},Z^{\top})$, we get
  \begin{align*}
    (\grad_Y X^\perp,Z)
    &=
    \varphi (Y, Z) + \frac{(X^\perp,Z^\perp)}{2(X^\perp, X^\perp)}Y(|X^\perp|^2)
    - \frac{(X^\perp,Y^\perp)}{2(X^\perp, X^\perp)}Z(|X^\perp|^2)
    \\ &=
    \varphi (Y, Z) + \frac{(X^\perp,Z^\perp)}{|X^\perp|}Y(|X^\perp|)
    - \frac{(X^\perp,Y^\perp)}{|X^\perp|}Z(|X^\perp|),
  \end{align*}
  as desired.
  The same approach gives the second result (for $X^{\top}$) provided we use the decomposition 
  \[
    Y = \frac{(Y,X^{\top})}{(X^{\top},X^{\top})}X^{\top} + Y_{1}, \quad Y_{1} \perp X^{\top},
  \]
  in place of $Y = Y^{\top} + Y^{\perp}$.

  Now note that since $\frac{|X^\perp|^2}{\varphi^2}$ is constant along $\mathfrak{F}$ we have 
  \begin{align*}
    &\frac{(X^\perp,Z^\perp)}{|X^\perp|}Y(|X^\perp|)
    - \frac{(X^\perp,Y^\perp)}{|X^\perp|}Z(|X^\perp|)
    \\ &= \frac{(X^\perp,Z^\perp)}{|X^\perp|}Y\left( \frac{|X^\perp|}{\varphi} \varphi \right)
    - \frac{(X^\perp,Y^\perp)}{|X^\perp|}Z\left( \frac{|X^\perp|}{\varphi} \varphi \right)
    \\ &= \frac{(X^\perp,Z^\perp)}{\varphi}Y\left( \varphi \right)
    - \frac{(X^\perp,Y^\perp)}{\varphi}Z\left( \varphi \right)
    \\ &+ \varphi\frac{(X^\perp,Z^\perp)}{|X^\perp|}Y^\perp\left( \frac{|X^\perp|}{\varphi}  \right)
    - \varphi \frac{(X^\perp,Y^\perp)}{|X^\perp|}Z^\perp\left( \frac{|X^\perp|}{\varphi} \right).
  \end{align*}
  The last line is zero since we can exchange $Y^\perp$ and $X^\perp$ in the first term and $Z^\perp$ and $X^\perp$ in the second term.
  Thus we get 
  \[
    (\grad_Y X^\perp,Z)
    =
    \varphi (Y,Z) + \frac{(X^\perp,Z)}{\varphi} Y(\varphi) - \frac{(X^\perp,Y)}{\varphi} Z(\varphi)
  \]
  providing the first result.
\end{proof}

\begin{remark}\label{rem:gradshort}
  It follows directly from Proposition \ref{prop:grad-of-xtop-xperp} that
  \[
    \grad_Y \frac{X^\perp}{\varphi} = Y - \frac{(X^\perp, Y)}{\varphi^2} \grad \varphi
  \]
  and
  \[
    \grad_Y \frac{X^\top}{|X^\top|} = - \frac{(X^\top, Y)}{|X^\top|^2} \grad |X^\top|
  \]
  as long as $|X^\top| \neq 0$.
\end{remark}

\begin{proposition}\label{prop:ric}
  Setting
  \[
    N^{\perp} = \frac{X^{\perp}}{|X^{\perp}|}, \quad N^{\top} = \frac{X^{\top}}{|X^{\top}|},
  \]
  we have
  \begin{align*}
    \overline{R}(Y, X^\perp, Y, X^\perp) &= - \frac{|X^\perp|^2}{\varphi} (\D_Y \D \varphi, Y) + \frac{(Y,X^\perp)^2}{\varphi} (\D_{\mathcal{N}^\perp} \D \varphi, \mathcal{N}^\perp)\\
    \overline{\Ric}(X^\perp, Y) &= - \frac{(X^\perp, Y)}{\varphi} \left( \overline{\Delta} \varphi - (\D_{\mathcal{N}^\perp} \D \varphi, \mathcal{N}^\perp) \right)\\
    \overline{R}(Y, X^\top, Y, X^\top) &= - |X^\top| (\D_Y \D \varphi, Y) + \frac{(Y,X^\top)^2}{|X^\top|} (\D_{\mathcal{N}^\top} \D \varphi, \mathcal{N}^\top)\\
    \overline{\Ric}(X^\top, Y) &= - \frac{(X^\top, Y)}{|X^\top|} \left( \overline{\Delta} |X^\top| - (\D_{\mathcal{N}^\top} \D |X^\top|, \mathcal{N}^\top) \right)
  \end{align*}
  as long as $|X^\top| \neq 0$.

  Additionally we have 
  \[
    \varphi \overline{\Ric}(X^\perp, Y) = (\nabla_Y \nabla \varphi, X^\perp) = (\nabla_{X^\perp} \nabla \varphi, Y) = 0
  \]
  for $Y \perp X^\perp$ as well as 
  \[
    |X^\top| \overline{\Ric}(X^\top, Y) = (\nabla_Y \nabla |X^\top|, X^\top) = (\nabla_{X^\top} \nabla |X^\top|, Y) = 0
  \]
  for $Y \perp X^\top$.
\end{proposition}

\begin{proof}
  We start with an application of Remark \ref{rem:gradshort}:
  \begin{align*}
    \left( \overline{R}(e_i, e_k) \frac{X^\perp}{\varphi}, e_\ell \right)
    &=
    \left( 
    \D_{i} \D_k \frac{X^\perp}{\varphi} - \D_{k} \D_i \frac{X^\perp}{\varphi}, e_\ell 
    \right)
    \\ &=
    \left( 
    \D_{i} \left( e_k - \frac{(X^\perp, e_k)}{\varphi^2} \D \varphi \right) 
    - 
    \D_{k} \left( e_i - \frac{(X^\perp, e_i)}{\varphi^2} \D \varphi \right), 
    e_\ell 
    \right)
    \\ &=
    \left( 
    \D_{k} \left( \frac{(X^\perp, e_i)}{\varphi^2} \D \varphi \right)
    - 
    \D_{i} \left( \frac{(X^\perp, e_k)}{\varphi^2} \D \varphi \right)  , 
    e_\ell 
    \right)
    \\ &=
    e_{k} \left( \frac{(X^\perp, e_i)}{\varphi^2} \right) e_\ell(\varphi)
    -
    e_{i} \left( \frac{(X^\perp, e_k)}{\varphi^2} \right) e_\ell(\varphi)
    \\ &+
    \frac{(X^\perp, e_i)}{\varphi^2} 
    \left( 
    \D_{k} \D \varphi, 
    e_\ell 
    \right)
    -
    \frac{(X^\perp, e_k)}{\varphi^2} 
    \left( 
    \D_{i} \D \varphi, 
    e_\ell 
    \right)\\
    &:= A_{1} +
    \frac{(X^\perp, e_i)}{\varphi^2} 
    \left( 
    \D_{k} \D \varphi, 
    e_\ell 
    \right)
    -
    \frac{(X^\perp, e_k)}{\varphi^2} 
    \left( 
    \D_{i} \D \varphi, 
    e_\ell 
    \right).
  \end{align*}
  Next, compute
  \begin{align*}
    A_{1} &=
    \frac{(\D_k X^\perp, e_i)}{\varphi^2} 
    -
    \frac{(\D_i X^\perp, e_k)}{\varphi^2} 
    +
    2 \frac{(X^\perp, e_k)}{\varphi^3}  e_i (\varphi)
    -
    2 \frac{(X^\perp, e_i)}{\varphi^3}  e_k (\varphi),
  \end{align*}
  noting that $A_{1} = 0$ when $i = k$.
  By Proposition \ref{prop:grad-of-xtop-xperp}, we have
  \begin{align*}
    (\D_k X^\perp, e_i) - (\D_i X^\perp, e_k) = 
    2(\D_k X^\perp, e_i) = 
    2\frac{e_k(\varphi)}{\varphi} (X^\perp, e_i)  - 2\frac{(X^\perp,e_k)}{\varphi} e_i(\varphi)
  \end{align*}
  and so $A_{1} = 0$.
  Consequently, there holds
  \begin{align*}
    \left( \overline{R}(e_i, e_k) \frac{X^\perp}{\varphi}, e_\ell \right)
    =
    \frac{(X^\perp, e_i)}{\varphi^2} 
    \left( 
    \D_{k} \D \varphi, 
    e_\ell 
    \right)
    -
    \frac{(X^\perp, e_k)}{\varphi^2} 
    \left( 
    \D_{i} \D \varphi, 
    e_\ell 
    \right).
  \end{align*}
  Substituting $e_i = e_\ell = Y$ and $e_k = X^\perp$ and then multiplying by $\varphi$ gives the first result.

  An identical computation shows that
  \begin{align*}
    \left( \overline{R}(e_i, e_k) \frac{X^\top}{|X^\top|}, e_\ell \right)
    =
    \frac{(X^\top, e_i)}{|X^\top|^2} 
    \left( 
    \D_{k} \D |X^\top|, 
    e_\ell 
    \right)
    -
    \frac{(X^\top, e_k)}{|X^\top|^2} 
    \left( 
    \D_{i} \D |X^\top|, 
    e_\ell 
    \right).
  \end{align*}
  Similar to above, substituting $e_i = e_\ell = Y$ and $e_k = X^\top$ and then multiplying by $|X^\top|$ gives us the third result.

  Taking traces, we get
  \begin{align*}
    \overline{\Ric}(e_k, X^\perp) &= \varphi g^{il}\left( \overline{R}(e_i, e_k) \frac{X^\perp}{\varphi}, e_\ell \right)
    =
    \frac{\left( 
      \D_{k} \D \varphi, 
      X^\perp
    \right)}{\varphi} 
    -
    \frac{(X^\perp, e_k)}{\varphi} 
    \overline{\Delta} \varphi\\
    \overline{\Ric}(e_k, X^\top) &= |X^\top| g^{il}\left( \overline{R}(e_i, e_k) \frac{X^\top}{|X^\top|}, e_\ell \right)
    =
    \frac{\left( 
      \D_{k} \D |X^\top|, 
      X^\top
    \right)}{|X^\top|} 
    -
    \frac{(X^\top, e_k)}{|X^\top|} 
    \overline{\Delta} |X^\top|.
  \end{align*}
  Next, using the Codazzi equations and that $h_{ij} = \frac{\varphi}{|X^\perp|} g_{ij}$ along $S_\alpha$, the covariant derivatives of $h_{ij}$ vanish and hence $\overline{R}_{ijkX^\perp}$ vanishes. For the same reason $\overline{R}_{ijkX^\top}$ also vanishes. 

  Using that we get that for any $Y$ tangent to $S_\alpha$
  \begin{align*}
    0 = \overline{\Ric}(Y,X^\perp) = \frac{(\D_Y \D \varphi, X^\perp)}{\varphi}
  \end{align*}
  which gives us that $(\D_Y \D \varphi, X^\perp) = 0$, similarly for any $Y$ tangent to $\mathfrak{F}$ we get
  \begin{align*}
    0 = \overline{\Ric}(Y,X^\top) = \frac{(\D_Y \D |X^\top|, X^\top)}{|X^\top|}
  \end{align*}
  giving us $(\D_Y \D |X^\top|, X^\top) = 0$.

  Plugging this back into the previous equations gives us 
  \begin{align*}
    \overline{\Ric}(Y, X^\perp) 
    &=
    \frac{\left( 
      \D_{Y} \D \varphi, 
      X^\perp
    \right)}{\varphi} 
    -
    \frac{(X^\perp, Y)}{\varphi} 
    \overline{\Delta} \varphi,
    \\ &=
    \frac{\left( 
      \D_{Y^\perp} \D \varphi, 
      X^\perp
    \right)}{\varphi} 
    +
    \frac{\left( 
      \D_{Y^\top} \D \varphi, 
      X^\perp
    \right)}{\varphi} 
    -
    \frac{(X^\perp, Y)}{\varphi} 
    \overline{\Delta} \varphi
    \\ &=
    -\frac{(X^\perp, Y)}{\varphi}
    \left(
    \overline{\Delta} \varphi
    -
    \left( 
    \D_{N^\perp} \D \varphi, 
    N^\perp
    \right)
    \right).
  \end{align*}
  And similarly for $X^\top$ we have 
  \begin{align*}
    \overline{\Ric}(Y, X^\top) 
    &=
    \frac{\left( 
      \D_{Y} \D |X^\top|, 
      X^\top
    \right)}{|X^\top|} 
    -
    \frac{(X^\top, Y)}{|X^\top|} 
    \overline{\Delta} |X^\top|
    \\ &=
    - \frac{(X^\top, Y)}{|X^\top|}
    \left( 
    \overline{\Delta} |X^\top|
    -
    \left( 
    \D_{N^\top} \D |X^\top|, 
    N^\top
    \right)
    \right).
  \end{align*}
\end{proof}

\begin{remark}
  We have 
  \begin{align*}
    (\D_Y \D \lambda, Z)
    &=
    2 
    (\D_Y \Lambda X^\perp, Z)
    \\ &=
    \frac{2}{\varphi^2} Y(\Lambda\varphi^2) 
    (X^\perp, Z)
    -
    \frac{4\Lambda}{\varphi} Y(\varphi) 
    (X^\perp, Z)
    +
    2\Lambda\varphi (Y,Z)
    \\ &+
    2 \Lambda\frac{Y(\varphi)(X^\perp, Z)}{\varphi} - 2\Lambda\frac{Z(\varphi) (X^\perp, Y)}{\varphi}
    \\ &=
    \frac{2}{\varphi^2} Y(\Lambda\varphi^2) 
    (X^\perp, Z)
    -
    \frac{2\Lambda}{\varphi} Y(\varphi) 
    (X^\perp, Z)
    +
    2\Lambda\varphi (Y,Z) - 2\Lambda\frac{Z(\varphi) (X^\perp, Y)}{\varphi}.
  \end{align*}
  By symmetry of the Hessian we have 
  \begin{align*}
    0 
    &= \left( \D_Y \D \lambda , Z \right) - \left( \D_Z \D \lambda, Y \right)
    = \left( \D_Y 2\Lambda X^\perp, Z \right) - \left( \D_Z 2\Lambda X^\perp, Y \right)
    \\ &= \frac{2}{\varphi^2} Y(\Lambda\varphi^2) 
    (X^\perp, Z)- \frac{2}{\varphi^2} Z(\Lambda\varphi^2) 
    (X^\perp, Y).
  \end{align*} 
  Plugging in $Y = X^\perp$ we get 
  \[
    X^\perp(\Lambda\varphi^2) 
    (X^\perp, Z) = Z(\Lambda\varphi^2) 
    (X^\perp, X^\perp)
  \]
  giving us 
  \[
    \D (\Lambda \varphi^2) = \frac{X^\perp(\Lambda\varphi^2)}{|X^\perp|^2} X^\perp.
  \]
\end{remark}

\section{Evolution Equations}\label{sec:evolution-equations}

In this section we show that under the flow \eqref{eq:main-flow} the volume is preserved and area is decreased.
We then compute the evolution equations for $\lambda = |X^{\perp}|^{2}/\varphi^{2}$, $u$ and $H$.
For simplicity, we do this first assuming $X$ is time-independent, noting that the equations in case $X$ is time-dependent follow easily.
To begin, we state and prove the following Minkowski identities.

\begin{proposition}
  Let $\Sigma \subset U$ be an embedded closed hypersurface.
  Then
  \begin{align*}
    \int_{\Sigma} Hu &= \int_{\Sigma} n\varphi\\
    \int_{N} H(n\varphi - Hu) 
    &=
    \frac{n}{(n-1)} \int_{N} u(\overline\Ric(\mathcal{N}^\perp, \mathcal{N}^\perp) - \overline\Ric(\nu, \nu))
    -\int_{N} \sum_{i < j} (\kappa_i - \kappa_j)^2u.
  \end{align*}
  \label{prop:minkowski-identities}
\end{proposition}

\begin{proof} 
We will follow the proof of \cite{kwong_extension_2014}, use orthonormal coordinates and let $X'$ denote the projection of $X$ onto $TN$, (not to be confused with $X^\top$).
Computing
\[
  \Div\left( f T_k(X') \right)
  =
  ( T_k(\nabla f), X')
  +
  f(\Div T_k)(X')
  +
  \frac{1}{2} f ( T_k^\flat, \iota^* (\mathcal{L}_X \overline{g}) )
  - 
  f ( T_k^\flat, h^{i}_{j} u ),
\]
we can choose $f = 1$, note that $\mathcal{L}_X \overline{g} = 2 \varphi \overline{g}$ and compute 
\[
  \Div\left( T_k(X') \right)
  =
  (\Div T_k)(X')
  +
  \varphi ( T_k^\flat, \overline{g} )
  - 
  ( T_k^\flat, h^{i}_{j} u ).
\]
Integrating then gives
\[
  \int_{N} ( T_k^\flat, h^{i}_{j} u ) =
  \int_{N} \Div T_k(X')
  +
  \int_{N} \varphi ( T_k^\flat, \overline{g} ).
\]

Next, using \cite[Lemma 2.2]{andrzejewski_newton_2010} for $k = 0$, we have 
\[
  T_0 = I, \quad \Div T_0 = 0
\]
and so 
\begin{align*}
  \int_{N} ( \delta_{ij}, h^{i}_{j} u ) &=
  \int_{N} \varphi ( \delta_{ij}, g_{ij} )\\
\end{align*}
which is equivalent to
\begin{align*}
  \int_{N} H u &= 
  \int_{N} n \varphi,
\end{align*}
thereby proving the first identity.

Now for the second identity, we use the $k = 1$ case, in this case we get (c.f. \cite[Lemma 3.1]{alias_constant_2006}) 
\begin{align*}
  T_1 &= H I - h^{i}_j, \\
  \Div T_1 X' &= -( (\overline R(\nu, e^{j}) e^{j}), X'), \quad\\
  \Tr(T_1) &= (n-1)H, \\ 
  ( T_1, h^i_j ) &= 2 H_2 .
\end{align*}
Using Proposition \ref{prop:ric} and Assumptions \ref{assumption2}(x), we get
\begin{align*}
  ( \overline R(\nu, e^{j}) e^{j}, X')
  &=
  \overline\Ric(\nu, e_j) (e_j, X') 
  = \overline\Ric(\nu, X) - u \overline\Ric(\nu, \nu)
  \\ &= \overline\Ric(\nu, X^\top + X^\perp) - u \overline\Ric(\nu, \nu)
  \\ &= \overline\Ric(\nu, X^\top) + \overline\Ric(\nu, X^\perp) - u \overline\Ric(\nu, \nu)
  \\ &= u^\top \overline\Ric(\mathcal{N}^\top, \mathcal{N}^\top) + u^\perp\overline\Ric(\mathcal{N}^\perp, \mathcal{N}^\perp) - u \overline\Ric(\nu, \nu)
  \\ &= u \overline\Ric(\mathcal{N}^\perp, \mathcal{N}^\perp) - u \overline\Ric(\nu, \nu),
\end{align*}
and so the equation simplifies to be 
\[
  2 \int_{N} H_2 u =
  (n-1) \int_{N} \varphi H
  -
  \int_{N} (u \overline\Ric(\mathcal{N}^\perp, \mathcal{N}^\perp) - u \overline\Ric(\nu, \nu)).
\]
We can now use this to get 
\begin{align*}
  &\int_{N} H(n\varphi - Hu) \\
  &=
  \int_{N} Hn\varphi - \frac{2n}{(n-1)}H_2u 
  + \frac{2n}{(n-1)}H_2u - H^2u\\
  &=
  \frac{n}{(n-1)}\int_{N} (n-1) H \varphi - 2 H_2u 
  + \int_{N}\frac{2n}{(n-1)}H_2u - H^2u\\
  &=
  \frac{n}{(n-1)} \int_{N} (u \overline\Ric(\mathcal{N}^\perp, \mathcal{N}^\perp) - u \overline\Ric(\nu, \nu)) + \int_{N}\frac{2n}{(n-1)}H_2u - H^2u\\
  &=
  \frac{n}{(n-1)}\int_{N} u(\overline\Ric(\mathcal{N}^\perp, \mathcal{N}^\perp) - \overline\Ric(\nu, \nu)) -\int_{N} \left(H^2 - \frac{2n}{n-1}H_2\right)u\\
  &=
  \frac{n}{(n-1)} \int_{N} u(\overline\Ric(\mathcal{N}^\perp, \mathcal{N}^\perp) - \overline\Ric(\nu, \nu))
  -\int_{N} \sum_{i < j} (\kappa_i - \kappa_j)^2u,
\end{align*}
thereby proving the second identity.

\end{proof}

\begin{corollary}
  Let $\Sigma_{0} \subset U$ be a embedded closed strictly starshaped hypersurface, suppose that, for some $T>0$, $\Sigma_{t}$ is a solution to \eqref{eq:main-flow} on $[0,T)$ with the $\Sigma_{t}$ strictly starshaped and let $A(t)$ be the surface area of $\Sigma_{t}$ and $V(t)$ the volume enclosed. %]
  Then
  \[
    V'(t) = 0 \qquad \text{ and } \qquad A'(t) \leq 0.
  \]
\end{corollary}

\begin{proof}
  It is standard to compute evolution equations of $A$ and $V$.
  For the flow \eqref{eq:main-flow}, we have
  \begin{align*}
    V'(t) &= \int_{\Sigma_{t}} (n \varphi - uH) d\sigma\\
    A'(t) &= \int_{\Sigma_{t}} (n \varphi - uH) H d\sigma.
  \end{align*}
  Using $u>0$ and Assumptions \ref{assumption2}(ix) we have $\overline\Ric(\mathcal{N}^{\perp},\mathcal{N}^{\perp}) - \overline\Ric(\nu,\nu) \leq 0$ and so Proposition \ref{prop:minkowski-identities} gives $V'(t) = 0$ and $A'(t) \leq 0$ immediately.
\end{proof}

\begin{theorem}[Evolution Equation for $\lambda$]
  \label{thm:evolution-equation-for-S}
  \[
    \partial_t \lambda - u \Delta_g \lambda  = -2\Lambda n\varphi u^\top - u\frac{2}{\varphi^2 |X^\perp|^2} X^\perp(\Lambda \varphi^2)(|X^\perp|^2 - (u^\perp)^2)
    +
    4u\frac{\Lambda }{\varphi}e_i(\varphi)(e_i,X^\perp)
  \]
\end{theorem}

\begin{proof}
  We have the following evolution equation for $\lambda$
  \begin{align*}
    \partial_t \lambda
    = (n \varphi - Hu) (\nu, \D \lambda)
    = 2 \Lambda  u^\perp (n \varphi - Hu)
  \end{align*}
  as well as 
  \begin{align*}
    \Delta_g \lambda
    &=
    \nabla_i \nabla_i \lambda
    =
    \nabla_i (2\Lambda (X^\perp, e_i))
    =
    2\Lambda  \nabla_i (X^\perp, e_i)
    +
    2(X^\perp, e_i) \nabla_i \Lambda 
    \\ &= 
    2\Lambda  (\nabla_i X^\perp, e_i)
    +
    2\Lambda  (X^\perp, \nabla_i e_i)
    +
    2\D_{X^\perp - u^\perp\nu} \Lambda 
    \\ &= 
    2\Lambda n \varphi
    -
    2\Lambda H u^\perp
    +
    2(\D \Lambda ,X^\perp - u^\perp\nu )
    \\ &= 
    2\Lambda n \varphi
    -
    2\Lambda H u^\perp
    +
    \frac{2}{\varphi^2}(\D (\Lambda \varphi^2),X^\perp - u^\perp\nu )
    -
    \frac{2}{\varphi^2}\Lambda (\D (\varphi^2),X^\perp - u^\perp\nu )
    \\ &= 
    2\Lambda n \varphi
    -
    2\Lambda H u^\perp
    +
    \frac{2}{\varphi^2|X^\perp|^2}X^\perp(\Lambda \varphi^2) (|X^\perp|^2 - (u^\perp)^2)
    -
    4\frac{\Lambda }{\varphi}e_i(\varphi)(e_i,X^\perp).
  \end{align*}
  Putting this altogether gives the desired result.
  \[
    \partial_t \lambda - u \Delta_g \lambda = -2\Lambda n\varphi u^\top - u\frac{2}{\varphi^2 |X^\perp|^2} X^\perp(\Lambda \varphi^2)(|X^\perp|^2 - (u^\perp)^2)
    +
    4u\frac{\Lambda }{\varphi}e_i(\varphi)(e_i,X^\perp)
  \]

\end{proof}

\begin{corollary}
  \label{cor:S-is-bounded}
  Suppose $\Sigma_{0}$ has evolution $\Sigma_{t}$ obeying the flow \eqref{eq:main-flow}.
  Then
  \[
    \min_{\Sigma_{0}} \lambda \leq \lambda \leq \max_{\Sigma_{0}} \lambda
  \]
  for all time $t$ the flow exists and all points on $\Sigma_{t}$.
\end{corollary}

\begin{proof}
  Using Theorem \ref{thm:evolution-equation-for-S} we will show that $\lambda$ satisfies $\p_{t}\lambda = u \Delta_{g}\lambda$ at critical points of $\lambda$ and hence may apply the maximum principle to obtain the result.
  Let $(p_{0},t_{0})$ be a critical point, where $p_{0} \in \Sigma_{t_{0}}$.
  Let $\left\{ e_{i} \right\}$ be a frame for $\Sigma_{t_{0}}$.
  Recalling $\overline\grad \lambda = 2\Lambda  X^{\perp}$, we have at $(p_{0},t_{0})$ that
  \[
    (\overline\grad \lambda, e_{i}) = 2\Lambda  (X^{\perp},e_{i}) = \frac{\varphi^{2} - X^{\perp}\varphi}{\varphi^{3}} (X^{\perp},e_{i}) = 0,
  \]
  and so $X^{\perp}$ is normal to $\Sigma_{t_{0}}$.
  This implies $u^{\top} = 0$ and $ u^{\perp} = |X^{\perp}|$, providing the desired result. 
\end{proof}

\begin{remark}
  By Corollary \ref{cor:S-is-bounded} we have that the evolution $\Sigma_{t}$ of $\Sigma_{0}$ under the flow \eqref{eq:main-flow} remains a compact set bounded by two leaves of the foliation $\mathfrak{F}$.
  Note also that this shows that the flow remains inside $U$.
\end{remark}

\begin{theorem}[Evolution Equation for $u$]
  \label{thm:evolution-equation-for-u}
  \begin{align*}
    \partial_{t} u - u\Delta_g u
    = n \varphi^2 - n X(\varphi) 
    - 2\varphi Hu + |A|^2 u^2 + 2nu\nu(\varphi)
    +
    u^2\overline{\Ric}(\nu,\nu)
    +
    H\left( X, \nabla u\right)
  \end{align*}
\end{theorem}
\begin{proof}
  First let us assume that at the point of evaluation $|X^\top| \neq 0$. 
  Then for $u$ we have the following time evolution:
  \begin{align*}
    \partial_t u 
    &=
    \left( \partial_t X, \nu \right)
    +
    \left( X, \partial_t \nu \right)
    =
    (n\varphi - Hu)\left( \D_\nu X, \nu \right)
    +
    \left( X, -\nabla(n \varphi - Hu)\right)
    \\ &=
    (n\varphi - Hu)\varphi
    -
    n \left( X, \nabla \varphi \right)
    +
    H\left( X, \nabla u\right)
    +
    u\left( X, \nabla H\right)
    \\ &=
    (n\varphi - Hu)\varphi
    -
    n X(\varphi) + nu\nu(\varphi)
    +
    H\left( X, \nabla u\right)
    +
    u\left( X, \nabla H\right).
  \end{align*}

  To express $\Delta_{g} u$, we treat $u^\perp$ and $u^\top$ separately.
  For $u^\perp$ we have
  \begin{equation*}
    \nabla_i (\nabla_i u^\perp)
    =
    \nabla_i (\nabla_i (X^\perp, \nu))
    =
    \nabla_i ((\nabla_i X^\perp, \nu))
    +
    \nabla_i ((X^\perp,\nabla_i \nu)).
  \end{equation*}
  To compute the two terms, we state and prove the following claim.

  \noindent
  \textit{Claim 1.}
    \begin{equation*}
      \nabla_i (\nabla_i X^\perp, \nu)
      =
      -\Ric(X^\perp, \nu) - n\nu(\varphi).
    \end{equation*}
    \begin{proof}[Proof of Claim 1]
    Using Proposition \ref{prop:grad-of-xtop-xperp}, to get the following terms
    \begin{align*}
      \nabla_i (\nabla_i X^\perp, \nu)
      &= 
      \nabla_i \left( 
      \frac{(X^\perp, \nu) e_i(\varphi) - (X^\perp, e_i) \nu(\varphi)}{\varphi}
      \right) 
      \\ &= 
      \frac{-e_i(\varphi)\left((X^\perp, \nu) e_i(\varphi) - (X^\perp, e_i) \nu(\varphi)\right)}{\varphi^2} \tag{1}
      \\ &+ 
      \frac{(\nabla_i X^\perp, \nu) e_i(\varphi)}{\varphi} \tag{2}
      \\ &+ 
      \frac{(X^\perp, \nabla_i \nu) e_i(\varphi)}{\varphi} \tag{3}
      \\ &+ 
      \frac{(X^\perp, \nu) (\nabla_i e_i, \grad \varphi)}{\varphi} \tag{4}
      \\ &+ 
      \frac{(X^\perp, \nu) (e_i, \nabla_i \grad \varphi)}{\varphi} \tag{5}
      \\ &- 
      \frac{(\nabla_i X^\perp, e_i) \nu(\varphi)}{\varphi} \tag{6}
      \\ &- 
      \frac{(X^\perp, \nabla_i e_i) \nu(\varphi)}{\varphi} \tag{7}
      \\ &-
      \frac{(X^\perp, e_i) (\nabla_i \nu, \grad \varphi)}{\varphi} \tag{8}
      \\ &- 
      \frac{(X^\perp, e_i) (\nu, \nabla_i \grad \varphi)}{\varphi} \tag{9}.
    \end{align*}
    We now show that many of these terms cancel.
    Note that term (2) after expanding using Proposition \ref{prop:grad-of-xtop-xperp} gives us 
    \[
      \frac{e_i(\varphi)\left((X^\perp, \nu) e_i(\varphi) - (X^\perp, e_i) \nu(\varphi)\right)}{\varphi^2}
    \]
    which cancels out term (1).
    Next term (3) simplifies to 
    \[
      \frac{h_{ik} (X^\perp, e_k) e_i(\varphi)}{\varphi},
    \]
    term (8) simplifies to 
    \[
      -\frac{h_{ik} (X^\perp, e_i) e_k(\varphi)}{\varphi}
    \]
    and since $h_{ik}$ is symmetric these two terms cancel.
    Next, term (4) and term (7) both simplify to 
    \[
      H \frac{(X^\perp, \nu) (\nu, \grad \varphi)}{\varphi},
    \]
    with opposing signs and so also cancel.

    Thus we are only left with terms (5), (6) and (9) and so we have 
    \[
      \nabla_i (\nabla_i X^\perp, \nu)
      =
      \frac{(X^\perp, \nu) (e_i, \nabla_i \grad \varphi)}{\varphi} 
      -
      \frac{(\nabla_i X^\perp, e_i) \nu(\varphi)}{\varphi}
      -
      \frac{(X^\perp, e_i) (\nu, \nabla_i \grad \varphi)}{\varphi}.
    \]
    Now notice that due to the fact that $X^\perp$ is conformal with factor $\varphi$ we get 
    \begin{align*}
      \frac{(\nabla_i X^\perp, e_i) \nu(\varphi)}{\varphi}
      =
      \frac{n \varphi \nu(\varphi)}{\varphi}
      =
      n \nu(\varphi),
    \end{align*}
    and so
    \[
      \nabla_i (\nabla_i X^\perp, \nu)
      =
      \frac{(X^\perp, \nu) (e_i, \nabla_i \grad \varphi)}{\varphi} 
      -
      \frac{(X^\perp, e_i) (\nu, \nabla_i \grad \varphi)}{\varphi}
      -
      n \nu(\varphi).
    \]
    Next, since $(X^\perp, e_i)e_i = X^\perp - (X^\perp, \nu) \nu$, we get 
    \[
      \frac{(X^\perp, e_i) (\nu, \nabla_i \grad \varphi)}{\varphi}
      =
      \frac{(\nu, \nabla_{X^\perp} \grad \varphi)}{\varphi}
      -
      \frac{(X^\perp, \nu) (\nu, \nabla_{\nu} \grad \varphi)}{\varphi},
    \]
    which when plugged back in gives us 
    \[
      \nabla_i (\nabla_i X^\perp, \nu)
      =
      \frac{(X^\perp, \nu) \Delta \varphi}{\varphi} 
      -
      \frac{(\nu, \nabla_{X^\perp} \grad \varphi)}{\varphi}
      -
      n \nu(\varphi).
    \]
    Now, orthogonally decomposing $\nu$ as 
    \[
      \nu = \frac{X^\perp}{|X^\perp|^2} (X^\perp, \nu) + \left( \nu - \frac{X^\perp}{|X^\perp|^2} (X^\perp, \nu) \right),
    \]
    we note the second term here is orthogonal to $X^\perp$, and so, by Proposition \ref{prop:ric} we have 
    \[
      (\nu, \nabla_{X^\perp} \grad \varphi)
      =
      \left((\nu, X^\perp) \frac{X^\perp}{|X^\perp|^2}, \nabla_{X^\perp} \grad \varphi\right)
      =
      (\nu, X^\perp) \left(N^\perp, \nabla_{N^\perp} \grad \varphi\right),
    \]
    which then gives
    \[
      \nabla_i (\nabla_i X^\perp, \nu)
      =
      \frac{(X^\perp, \nu) \Delta \varphi}{\varphi} 
      -
      \frac{(X^\perp, \nu) (N^\perp, \nabla_{N^\perp} \grad \varphi)}{\varphi}
      -
      n \nu(\varphi).
    \]
    By another use of Proposition \ref{prop:ric} we get
    \[
      \nabla_i (\nabla_i X^\perp, \nu)
      =
      - \Ric(X^\perp, \nu)
      -
      n \nu(\varphi),
    \]
    as desired. 
  \end{proof}
  Next, for the term $\nabla_i (X^\perp,\nabla_i \nu)$, we have
  \[
    \nabla_i (X^\perp, \nabla_i \nu) 
    =
    \nabla_i (X^\perp, h_{ik} e_k) 
    =
    h_{ik} (\nabla_i X^\perp, e_k) 
    +
    h_{ik} (X^\perp, \nabla_i e_k) 
    +
    h_{ik,i} (X^\perp, e_k).
  \]
  Now note that $h_{ik}$ is symmetric and the symmetrization of $(\nabla_i X^\perp, e_k)$ is $\varphi (e_i, e_k)$.
  Therefore, the first term becomes $\varphi H$ after summation. The second term gives us 
  \[
    h_{ik} (X^\perp, - h_{ik} \nu) = -|A|^2 u^\perp.
  \]
  Finally for the third term we use the Codazzi equation to get
  \[
    h_{ik,i} (X^\perp, e_k)
    =
    (h_{ii,k} + \overline{R}(e_i,e_k,e_i,\nu)) (X^\perp, e_k).
  \]
  Using$(X^\perp, e_i)e_i = X^\perp - (X^\perp, \nu) \nu$ we get
  \[
    \overline{R}(e_i,e_k,e_i,\nu) (X^\perp, e_k)
    =
    \overline{R}(e_i,X^\perp,e_i,\nu) 
    -
    u^\perp \overline{R}(e_i,\nu,e_i,\nu) 
    =
    \overline\Ric(X^\perp, \nu)
    -
    u^\perp \overline\Ric(\nu, \nu)
  \]
  and so
  \[
    \nabla_i (X^\perp, \nabla_i \nu) 
    =
    \varphi H + (X^\perp, \grad H) - |A|^2 u^\perp + \overline\Ric(X^\perp, \nu) - u^\perp \overline\Ric(\nu,\nu).
  \]
  combining this with Claim 1 we get
  \[
    \Delta u^\perp 
    =
    \varphi H + (X^\perp, \grad H) -|A|^2 u^\perp - n \nu (\varphi) - u^\perp \overline\Ric(\nu,\nu).
  \]

  Now we deal with $u^\top$.
  Although we will see that the treatment is similar, we must first assume that $|X^\top(p)| \neq 0$ where the following computations are evaluated at.
  At such a point we have
  \begin{equation*}
    \nabla_i (\nabla_i u^\top)
    =
    \nabla_i (\nabla_i (X^\top, \nu))
    =
    \nabla_i ((\nabla_i X^\top, \nu))
    +
    \nabla_i ((X^\top,\nabla_i \nu)).
  \end{equation*}
  
  \noindent
  \textit{Claim 2.}
    \[
      \nabla_i (\nabla_i X^\top, \nu)
      =
      - \overline\Ric(X^\top, \nu).
    \]

    \begin{proof}[Proof of Claim 2]
    First we substitute using Proposition \ref{prop:grad-of-xtop-xperp} to get the following terms
    \begin{align*}
      \nabla_i (\nabla_i X^\top, \nu)
      &= 
      \nabla_i \left( 
      \frac{(X^\top, \nu) e_i(|X^\top|) - (X^\top, e_i) \nu(|X^\top|)}{|X^\top|}
      \right) 
      \\ &= 
      \frac{-e_i(|X^\top|)\left((X^\top, \nu) e_i(|X^\top|) - (X^\top, e_i) \nu(|X^\top|)\right)}{|X^\top|^2} \tag{1}
      \\ &+ 
      \frac{(\nabla_i X^\top, \nu) e_i(|X^\top|)}{|X^\top|} \tag{2}
      \\ &+ 
      \frac{(X^\top, \nabla_i \nu) e_i(|X^\top|)}{|X^\top|} \tag{3}
      \\ &+ 
      \frac{(X^\top, \nu) (\nabla_i e_i, \grad |X^\top|)}{|X^\top|} \tag{4}
      \\ &+ 
      \frac{(X^\top, \nu) (e_i, \nabla_i \grad |X^\top|)}{|X^\top|} \tag{5}
      \\ &- 
      \frac{(\nabla_i X^\top, e_i) \nu(|X^\top|)}{|X^\top|} \tag{6}
      \\ &- 
      \frac{(X^\top, \nabla_i e_i) \nu(|X^\top|)}{|X^\top|} \tag{7}
      \\ &-
      \frac{(X^\top, e_i) (\nabla_i \nu, \grad |X^\top|)}{|X^\top|} \tag{8}
      \\ &- 
      \frac{(X^\top, e_i) (\nu, \nabla_i \grad |X^\top|)}{|X^\top|} \tag{9}.
    \end{align*}
    In the exact same manner as before we first note that term (2) after expanding using Proposition \ref{prop:grad-of-xtop-xperp} again, gives us 
    \[
      \frac{e_i(|X^\top|)\left((X^\top, \nu) e_i(|X^\top|) - (X^\top, e_i) \nu(|X^\top|)\right)}{|X^\top|^2}
    \]
    which cancels out term (1).
    Next, term (3) simplifies to 
    \[
      \frac{h_{ik} (X^\top, e_k) e_i(|X^\top|)}{|X^\top|}
    \]
    and term (8) simplifies to 
    \[
      -\frac{h_{ik} (X^\top, e_i) e_k(|X^\top|)}{|X^\top|}
    \]
    and, since $h_{ik}$ is symmetric, these terms cancel.
    Next, term (4) and term (7) both simplify to 
    \[
      H \frac{(X^\top, \nu) (\nu, \grad |X^\top|)}{|X^\top|},
    \]
    except with opposing signs and so also cancel.
    Finally, term (6) has $(\grad_i X^\top, e_i)$ which is zero since $X^\top$ is isometric. 
    We are thus left with 
    \[
      \nabla_i (\nabla_i X^\top, \nu) = 
      \frac{(X^\top, \nu) (e_i, \nabla_i \grad |X^\top|)}{|X^\top|}
      -
      \frac{(X^\top, e_i) (\nu, \nabla_i \grad |X^\top|)}{|X^\top|}
    \]
    which by $(X^\perp, e_i)e_i = X^\perp - (X^\perp, \nu) \nu$ gives us 
    \[
      \nabla_i (\nabla_i X^\top, \nu) = 
      \frac{(X^\top, \nu) \Delta |X^\top|}{|X^\top|}
      -
      \frac{(\nu, \nabla_{X^\top} \grad |X^\top|)}{|X^\top|}
    \] 
    and again by using orthogonal decomposition of $\nu$ and employing Proposition \ref{prop:ric} twice we get 
    \[
      \nabla_i (\nabla_i X^\top, \nu) = 
      \frac{(X^\top, \nu) \Delta |X^\top|}{|X^\top|}
      -
      \frac{(X^\top, \nu)(N^\top, \nabla_{N^\top} \grad |X^\top|)}{|X^\top|}
      =
      -\overline\Ric(X^\top, \nu),
    \]
    as desired.
  \end{proof}
  For the term $\nabla_i (X^\top,\nabla_i \nu)$, everything is the same as above except that now the symmetrization of $(\nabla_i X^\top, e_k)$ is zero and so that term goes away.
  We thus get 
  \[
    \nabla_i (X^\top, \nabla_i \nu) 
    =
    (X^\top, \grad H) - |A|^2 u^\top + \overline\Ric(X^\top, \nu) - u^\top \overline\Ric(\nu,\nu).
  \]
  Combining this with Claim 2 gives us 
  \[
    \Delta_{g} u^\top 
    =
    (X^\top, \grad H) - |A|^2 u^\top - u^\top \overline\Ric(\nu,\nu)
  \]

  We now handle the case $|X^{\top}|=0$.
  Thus, set $\mathcal{O} = \{ p : |X^\top(p)| = 0\}$.
  If $p$ is in the interior of $\mathcal{O}$ then all the derivatives of $X^\top$ vanish at $p$ and so the above equation for $\Delta_{g} u^{\top}$ trivially holds. 
  On the other hand, if $p$ is on the boundary of $\mathcal{O}$, then we can pick a sequence of points $p_k$ with $\lim_{k \to \infty} p_k = p$ for which $|X^\top(p_{k})| \neq 0$.
  Since the above equation holds for $p_{k}$ and the equation is equality of two continuous functions, if it holds at each $p_k$ it must also at the limit point and so we still have the above equation for $\Delta_{g} u^{\top}$.

  Combining the equations for $\Delta_{g}u^{\top}$ and $\Delta_{g}u^{\perp}$ gives
  \[
    \Delta_g u = (X, \nabla H) - u \overline{\Ric}(\nu,\nu) - |A|^2 u - n \nu(\varphi) + \varphi H.
  \]
  Combining this equation with the equation for $\p_{t} u$ gives the desired result.

\end{proof}

\begin{theorem}[Evolution Equation for $H$]
  \label{thm:evolution-equation-for-H}

  \begin{align*}
    (\partial_t - u \Delta_g)H 
    &= 2 (\nabla H, \nabla u) + H(X,\nabla H) - \varphi (H^2 - n|A|^2) 
    \\ &+ n\left( \left( \D_{\nu} \D \varphi, \nu  \right) - (\D_{\mathcal{N}^\perp} \D \varphi, \mathcal{N}^\perp) \right) + n\varphi \left( \overline{\Ric}(\mathcal{N}^\perp,\mathcal{N}^\perp) - \overline{\Ric}(\nu,\nu) \right)
  \end{align*}
\end{theorem}

\begin{proof}
  Under a flow $F' = f \nu$, it is standard to compute (c.f. \cite{guan_volume_2018}) 
  \[
    \partial_t H = - \Delta_g f - f |A|^2 - f \overline{\Ric}(\nu,\nu)
  \]
  where $f$ is the speed function along the flow. 
  Plugging in our speed function $f = n \varphi - Hu$ gives us 
  \begin{align*}
    \partial_t H 
    &= - \Delta_g (n \varphi - Hu) - f |A|^2 - f \overline{R}(\nu,\nu)
    \\ &= - n\Delta_g\varphi + u \Delta_g H + H \Delta_g u + 2 (\nabla H, \nabla u)  - n \varphi |A|^2 + Hu |A|^2 - n\varphi \overline{\Ric}(\nu,\nu) 
    \\ &+ Hu \overline{\Ric}(\nu,\nu).
  \end{align*}
  Plugging in our expression for $\Delta_g u$ gives us 
  \begin{align*}
    (\partial_t - u \Delta_g)H 
    &= - n\Delta_g\varphi + 2 (\nabla H, \nabla u) + H(X,\nabla H) - Hn\nu(\varphi) - \varphi (H^2 - n|A|^2) 
    \\ &- n\varphi \overline{\Ric}(\nu,\nu).
  \end{align*}
  Next note $\Delta_g \varphi = \overline{\Delta} \varphi - \left( \D_{\nu} \D \varphi, \nu  \right) - H \nu(\varphi)$ and so 
  \begin{align*}
    (\partial_t - u \Delta_g)H 
    &= -n\overline{\Delta} \varphi + n\left( \D_{\nu} \D \varphi, \nu  \right) + 2 (\nabla H, \nabla u) + H(X,\nabla H) - \varphi (H^2 - n|A|^2) 
    \\ &- n\varphi \overline{\Ric}(\nu,\nu).
  \end{align*}
  Lastly, from Proposition \ref{prop:ric} we have $\varphi \overline{\Ric}(N^\perp, N^\perp) = \overline{\Delta} \varphi - (\D_{N^\perp} \D \varphi, N^\perp)$ and so we get 
  \begin{align*}
    (\partial_t - u \Delta_g)H 
    &= 2 (\nabla H, \nabla u) + H(X,\nabla H) - \varphi (H^2 - n|A|^2) 
    \\ &+ n\left( \left( \D_{\nu} \D \varphi, \nu  \right) - (\D_{N^\perp} \D \varphi, N^\perp) \right) + n\varphi \left( \overline{\Ric}(N^\perp,N^\perp) - \overline{\Ric}(\nu,\nu) \right).
  \end{align*}
\end{proof}

\begin{corollary}
  \label{cor:H-bound}
  There exist absolute constants $a,b$ depending only on the initial hypersurface such that
  \[
    \norm{H}_{\oo} \leq a+bt
  \]
  for all time.
\end{corollary}

\begin{proof}
  By Theorem \ref{thm:evolution-equation-for-H} we have at any critical point of $H$ there holds 
  \begin{align*}
    (\partial_t - u \Delta_g)H 
    &=  - \varphi (H^2 - n|A|^2) + n\left( \left( \D_{\nu} \D \varphi, \nu  \right) - (\D_{\mathcal{N}^\perp} \D \varphi, \mathcal{N}^\perp) \right) \\
    &+ n\varphi \left( \overline{\Ric}(\mathcal{N}^\perp,\mathcal{N}^\perp) - \overline{\Ric}(\nu,\nu) \right)
    \\ &< n\left( \left( \D_{\nu} \D \varphi, \nu  \right) - (\D_{\mathcal{N}^\perp} \D \varphi, \mathcal{N}^\perp) \right)
  \end{align*}
  thus since the above is uniformly bounded above inside any compact set then by standard parabolic maximum principle we get the desired result.
\end{proof}

\begin{corollary}
  The flow \eqref{eq:main-flow} with strictly starshaped initial data exists for any finite time.
\end{corollary}

One might be tempted to continue the approach of Li-Pan in \cite{jiayu_isoperimetric_2023} to prove existence for $t = \infty$. However, if we define $\phi = H + \frac{|X^\perp|^2}{\varphi^2}$ to be our test function, then we get the following evolution equation 
\begin{align*}
  (\partial_t - u \Delta_g)\phi 
    &= 2 (\nabla \phi, \nabla u) + H(X,\nabla \phi) - \varphi (H^2 - n|A|^2) 
    \\ &+ n\left( \left( \D_{\nu} \D \varphi, \nu  \right) - (\D_{\mathcal{N}^\perp} \D \varphi, \mathcal{N}^\perp) \right) + n\varphi \left( \overline{\Ric}(\mathcal{N}^\perp,\mathcal{N}^\perp) - \overline{\Ric}(\nu,\nu) \right)
    \\ &- 2\left(\nabla \frac{|X^\perp|^2}{\varphi^2}, \nabla u\right)
    - H\left(\nabla \frac{|X^\perp|^2}{\varphi^2}, X\right)
    \\ &-2\Lambda n\varphi u^\top - u\frac{2}{\varphi^2 |X^\perp|^2} X^\perp(\Lambda \varphi^2)(|X^\perp|^2 - (u^\perp)^2)
    +
    4u\frac{\Lambda }{\varphi}e_i(\varphi)(e_i,X^\perp).
\end{align*}
From here we compute 
\begin{align*}
  &- 2\left(\nabla \frac{|X^\perp|^2}{\varphi^2}, \nabla u\right)
  =
  - 4 \Lambda  (X^\perp, e_i) e_i(u)
  \\ &=
  - 4 \Lambda  (X^\perp, e_i) (X,e_j) h_{ij}
  - 4 \Lambda  (X^\perp, e_i) \left( \nabla_{i} X^\perp, \nu \right) 
  - 4 \Lambda  (X^\perp, e_i) \left( \nabla_{i} X^\top, \nu \right).
\end{align*}  
The second term here simplifies as 
\[
  4 \Lambda  \left( |X^\perp|^2 - (u^\perp)^2 \right) \frac{\nu(\varphi)}{\varphi}
  - 4 \Lambda  (X^\perp, e_i) \frac{e_i(\varphi)}{\varphi} u^\perp,
\]
and the third term simplifies as 
\begin{align*}
  &- 4 \Lambda  u^\perp u^\top \frac{\nu(|X^\top|)}{|X^\top|}
  - 4 \Lambda  (X^\perp, e_i) \frac{e_i(|X^\top|)}{|X^\top|} u^\top
  \\ &- 4 \Lambda  u^\perp u^\top \frac{\nu(|X^\top|)}{|X^\top|}
  - 4 \Lambda  \frac{X^\top(|X^\top|) - u^\perp \nu(|X^\top|)}{|X^\top|} u^\top
  = 
  - 4 \Lambda  \frac{X^\top (|X^\top|)}{|X^\top|} u^\top.
\end{align*}
Next we simplify the numerator as 
\begin{align*}
  X^\perp (|X^\top|) 
  &= \frac{X^\perp(|X^\top|^2)}{2 |X^\top|}
  = \frac{2(X^\top, \partial_{X^\perp} X^\top)}{2 |X^\top|}
  = \frac{(X^\top, \partial_{X^\perp} X^\top)}{|X^\top|}
\end{align*}
and since $X^\top$ is a Killing vector field we can use the antisymmetry of its covariant derivative to get
\[
  (X^\top, \partial_{X^\perp} X^\top)
  =
  -(X^\perp, \partial_{X^\top} X^\top)
  =
  (\partial_{X^\top} X^\perp, X^\top)
  = \varphi |X^\top|^2
\]
and so the third term becomes 
\[
  - 4 \Lambda  \frac{X^\top (|X^\top|)}{|X^\top|} u^\top
  = 
  - 4 \Lambda  \varphi u^\top.
\]

Next we compute
\[
  - H\left(\nabla \frac{|X^\perp|^2}{\varphi^2}, X\right)
  - 2 H \Lambda  (X^\perp, e_i) (X, e_i)
  = - 2 H \Lambda  \left( |X^\perp|^2 - (u^\perp)^2 \right)
  + 2 H \Lambda  u^\perp u^\top.
\]

Plugging these all back in we get 
\begin{align*}
  (\partial_t - u \Delta_g)\phi 
  &= 2 (\nabla \phi, \nabla u) + H(X,\nabla \phi) - \varphi (H^2 - n|A|^2) 
  \\ &+ n\left( \left( \D_{\nu} \D \varphi, \nu  \right) - (\D_{\mathcal{N}^\perp} \D \varphi, \mathcal{N}^\perp) \right) + n\varphi \left( \overline{\Ric}(\mathcal{N}^\perp,\mathcal{N}^\perp) - \overline{\Ric}(\nu,\nu) \right)
  \\ &- 4 \Lambda  (X^\perp, e_i) (X,e_j) h_{ij}
  - 2 H \Lambda  \left( |X^\perp|^2 - (u^\perp)^2 \right)
  + 2 H \Lambda  u^\perp u^\top
  \\ &+ 4 \Lambda  \left( |X^\perp|^2 - (u^\perp)^2 \right) \frac{\nu(\varphi)}{\varphi}
  - (4 + 2n) \Lambda  \varphi u^\top
  \\ &- u\frac{2}{\varphi^2 |X^\perp|^2} X^\perp(\Lambda \varphi^2)(|X^\perp|^2 - (u^\perp)^2)
  +
  4u^\top \frac{\Lambda }{\varphi}e_i(\varphi)(e_i,X^\perp)
\end{align*}
If we assume we are at a maximal point then the gradient of $\phi$ vanishes and then we can rearrange the rest of the terms into 
\begin{align*}
  (\partial_t - u \Delta_g)\phi 
  &< - \varphi (H^2 - n|A|^2) 
  + 2 (c_1 - 2\lambda_1 - H) \Lambda  \left( |X^\perp|^2 - (u^\perp)^2 \right) + 2H\Lambda  u^\perp u^\top
  \\ &
  - 4 \Lambda  (X^\perp, e_i) (X^\top,e_j) h_{ij}
  - (4 + 2n)\Lambda \varphi u^\top
  +
  4u^\top \frac{\Lambda }{\varphi}e_i(\varphi)(e_i,X^\perp),
\end{align*}
here $c_i$ are just some constants and $\lambda_1$ is the smallest eigenvalue of $H$. The terms on the top row can be dealt with exactly as in \cite{jiayu_isoperimetric_2023}, however, all the other terms are difficult to bound or get a sign on, thus this approach seems illadivsed.

\section{Proof of Theorem \ref{thm:main}}\label{sec:main-proof}

We now relax the assumption that $X$ is constant and set $u(t) = \left( X^\perp + \Xi(t/T_0) X^\top, \nu \right)$.
Let $X^{\top} (t) = \Xi(t/T_{0}) X^{\top}$.
We notice that none of the evolution equations change, except for that of $u$, where we have 
\begin{align*}
  \partial_t u 
  &=
  \partial_t (X(t), \nu)
  \\ &=
  ((n \varphi - Hu) \D_\nu X(t), \nu) + \frac{1}{T_0}\Xi'(t/T_0)(X^\top, \nu) + (X(t), -\nabla(n \varphi - Hu))
  \\ &= 
  (n\varphi - Hu)\varphi
  -
  n X(t)(\varphi) + nu\nu(\varphi)
  +
  H\left( X(t), \nabla u\right)
  +
  u\left( X(t), \nabla H\right)
  \\ &+
  \frac{1}{T_0}\Xi'(t/T_0)(X^\top, \nu).
\end{align*}
This then gives us 
\begin{align*}
  \partial_t u - u\Delta_g u
  &= n \varphi^2 - n X(t)(\varphi) 
  - 2\varphi Hu + |A|^2 u^2 + 2nu\nu(\varphi)
  +
  u^2\overline{\Ric}(\nu,\nu)
  \\ &+
  H\left( X(t), \nabla u\right)
  + 
  \frac{1}{T_0}\Xi'(t/T_0) (X^\top, \nu)
\end{align*}
Since $0 \leq \Xi \leq 1$ and $\Lambda  \varphi^{3} > 0$, Assumptions \ref{assumption2}(vi) gives
\[
  \Lambda  \varphi^{3} - X^{\top}(t)\varphi >0
\]
everywhere in $U$ for all $t$.
Then since the flow is contained in a compact set $\mathcal{K} \subset U$ by Corollary \ref{cor:S-is-bounded} we get
\[
  \Lambda  \varphi^{3} - X^{\top}(t)\varphi > c > 0 
\]
for some constant $c$.
Then by increasing $T_0$ we can guarantee that
\[
  \left|\frac{1}{T_0}\Xi'(t/T_0) X^{\top}\right| < nc - \e
\]
for some small $\e>0$, for all $t$, and everywhere on $\mathcal{K}$.
We then conclude that at a critical point of $u$ we have
\begin{align*}
  (\p_{t} - u \Delta_{g}) u &> nc + nH^{2} u^{2} - 2\varphi H u + u \left( u \overline{\Ric}(\nu,\nu) + 2 n \nu(\varphi) \right) - (nc - \e)\\
  &> \e + nH^{2} u^{2} - 2\varphi H u - Cu.
\end{align*}
If $u < \e \left( 3 \max(1,C,2\varphi H) \right)^{-1}$, then
\begin{align*}
  (\p_{t} - u \Delta_{g})u &> \e + n H^{2} u^{2} - \frac{\e}{3} - \frac{\e}{3}\\
  & > \frac{\e}{3} + n H^{2} u^{2}.
\end{align*}
It follows from Corollary \ref{cor:H-bound} that, at any minimum point of $u$, $u$ cannot be decreasing whenever $u < \frac{1}{A + Bt}$ for some constants $A,B>0$.
Thus $u>0$ for any finite time.

Lastly, running the flow until $\Xi(t) = 0$ gives a surface which is starshaped with respect to $X^{\perp}$ only and therefore Theorem \ref{intro-thm:li-pan} may be applied.

\section{Convergence to a leaf}\label{sec:convergence}

Finally we prove the limit manifold $\Sigma_\infty$ of our flow is always a leaf $S_\alpha$ of the foliation $\mathfrak{F}$. 
Note that it is enough to prove $u^\perp = |X^\perp|$ at every point. For this section we will assume we are flowing `at infinity' and so $\Xi = 0$ and $u = u^\perp$.

Let $F$ be a solution to the flow \eqref{eq:main-flow} and let $F_{n} : \Sigma \times [0,1] \to U$ be defined by $F_{n}(t) = F(n+t)$ for $t \in [0,1]$.
Assuming that all higher derivatives of $F$ are sufficiently are bounded we conclude by the Arzel\`a-Ascoli theorem that $\left\{ F_{n} \right\}$ has a convergent subsequence, also denoted by $F_{n}$, with some limit $F_{\oo}:\Sigma \times [0,1] \to U$.
Showing that $\left\{ \p_{t} F_{n} \right\}$ is uniformly convergent is then enough to conclude that $\p_{t} F_{n} \to \p_{t} F_{\oo}$ and hence that $F_{\oo}$ solves the flow \eqref{eq:main-flow} with initial data $F_{\oo}(\Sigma, 0)$
Then we can repeatedly apply the Arzel\'a-Ascoli theorem to obtain the desired convergence.

To show that the higher derivatives of $F$ are bounded we can reparemetrize $\Sigma$ so that $F$ is a graph flow of $\lambda$ over a leaf $S_\alpha$, we then know that if we get bounds on all derivatives of $\lambda$ then those give us bounds on all derivatives of $F$. See details of this approach in the Appendix.

We write $\Sigma_{\oo}(t)$ for $F_{\oo}(\Sigma,t)$.
Now consider any positive continuous quantity $Q(F(t)) : [0,\infty) \to \R^+$ depending on the embedding that is non-increasing along this flow, we have then that its subsequential limit exists and satisfies
  \begin{align*}
    Q(F_\infty(1)) &= \lim_{n \to \infty} Q(F_n(1)) = \lim_{n \to \infty} Q(F(n+1)) \\ Q(F_\infty(0)) &= \lim_{n \to \infty} Q(F_n(0)) = \lim_{n \to \infty} Q(F(n))
  \end{align*}
  and so we have $Q(F_\infty(0)) = Q(F_\infty(1))$.

  Consider the area $A(t)$ along the flow. It is a positive continuous non-increasing quantity, and so it must be constant on $F_\infty$. We then get that $\partial_{t} A(F_\infty(t)) = 0$.
  The Newton-McLaurin inequality gives that $\Sigma_\infty(t)$ is totally umbilical for all $t$ in $[0,1]$. We additionally get that $\overline{Ric}(\mathcal{N}^\perp, \mathcal{N}^\perp) = \overline{Ric}(\nu_\infty(t), \nu_\infty(t))$ for all $t$ in $[0,1]$.
  See \cite{jiayu_isoperimetric_2023} for more details on these kinds of arguments for $A(t)$.

  Now consider the Gauss-Codazzi equations on this submanifold: under some orthonormal frame $e_i$ we get 
  \begin{align*}
    h_{ik,i} 
    &= h_{ii,k} + \overline{R}_{\nu_\infty iki}
    \\ h_{kk,k} - H_{k} &= \overline{\Ric}(\nu_\infty, e_k)
    \\ \frac{H_k}{n} - H_{k} &= \overline{\Ric}(\nu_\infty, e_k)
    \\ -\frac{H_k(n-1)}{n} &= \overline{\Ric}(\nu_\infty, e_k).
  \end{align*}
  Now we write $(\overline{\Ric}^\sharp)^i_j = g^{ik}\overline{\Ric}_{kj}$, note that by assumption $\mathcal{N}^\perp$ is an eigenvector of minimal eigenvalue for $\overline{\Ric}^\sharp$, but since $\overline{\Ric}(\mathcal{N}^\perp, \mathcal{N}^\perp) = \overline{\Ric}(\nu_\infty, \nu_\infty)$ then $\nu_\infty$ must be an eigenvector of the same eigenvalue.
  We thus have $\overline{\Ric}(\nu_\infty, e_k) = (\overline{\Ric}^\sharp(\nu_\infty), e_k) = (c \nu_\infty, e_k) = 0$ where $c$ is the eigenvalue.
  We thus get $H_k = 0$ for any $e_k$ and so $H$ must be constant and so $\Sigma_\infty$ is a surface of constant mean curvature (CMC).

  Next we note that for the same reason $\max_{\Sigma_\infty(t)} \lambda$ is also constant and we can use this fact to find a maximal stationary point of the flow at any time $t_0$ as follows. 
  Let $Z$ be the set of all maximal points of $\lambda$ along $\Sigma_\infty(t_0)$.
  For sake of contradiction, assume that $Z$ contains no stationary points of the flow.
  Then by our evolution equations at any point of $Z$ we must have $\partial_t \lambda = 2 \Lambda u (n \varphi - Hu) \leq 0$ and since $(n \varphi - Hu)$ is nonzero and $\lambda$ cannot increase at a maximal point, $\partial_{t} \lambda$ must be negative and so we get strictness, $\partial_t \lambda < 0$. 

  Now $Z$ is a closed set of a compact manifold $\Sigma_\infty(t_0)$ and thus is compact.
  We can thus pick $\varepsilon > 0$ such that $\partial_t \lambda < -\varepsilon$ uniformly along $Z$. Next, by continuity and compactness of $Z$, we can take a neighborhood $U$ in $\Sigma_\infty(t_0)$ which contains $Z$ and which is small enough so that $\partial_t \lambda < - \frac{\varepsilon}{2}$ uniformly over $U$.

  Next we know that $\Sigma_\infty(t_0) \setminus U$ is closed and thus compact.
  Therefore the image $\lambda(\Sigma_\infty(t_0) \setminus U)$ is compact and thus contains all its limit points, thus since it does not contain the image of any points with value $L$ we have some $\delta > 0$ so that $\lambda(p) < L - \delta$ for all $p \in \Sigma_\infty(t_0) \setminus U$.
  Then since $\Sigma_\infty(t_0)$ is compact we know that $|\partial_t \lambda| < B$ uniformly along $\Sigma_\infty(t_0)$ for some constant $B$. 

  Finally, considering the submanifold $\Sigma_\infty(t_0 + \Delta t)$, we must have for any point $p \in U$
  \[
    \lambda(p) < L - \Delta t \frac{\varepsilon}{2} + O((\Delta t)^2)
  \]
  and for any point in $p \in \Sigma_\infty \setminus U$
  \[
    \lambda(p) < L - \delta + \Delta t B + O((\Delta t)^2).
  \]
  Thus by choosing $\Delta t$ small enough we get that $\max_{\Sigma_\infty(t_0 + \Delta t)} \lambda < \max_{\Sigma_\infty(t)} \lambda - \frac{\varepsilon}{2} \Delta t$. This contradicts the fact that $\max_{\Sigma_\infty(t)} \lambda$ is constant and so we must have at least one stationary point in $Z$.

  Now let $p$ be such a stationary point.
  At any maximal point of $\lambda$ we have $u = |X^\perp|$ and so 
  \[
    H(p) = \frac{n \phi}{|X^\perp|} = n \lambda^{-1/2} = n L^{-1/2}
  \]
  and since we know $\Sigma_\infty$ is CMC we have for an arbitrary point $q$ 
  \[
    n \varphi - Hu = n \varphi - n L^{-1/2} u 
    \geq n \varphi - n \lambda^{-1/2} u 
    = n \varphi \left( 1 - \frac{u}{|X^\perp|} \right).
  \]
  Note that this expression is everywhere non-negative, hence since we know 
  \[
    \int_{\Sigma_\infty} n \varphi - Hu = 0
  \]
  we get $1 - \frac{u}{|X^\perp|} = 0$ everywhere and hence $u = |X^\perp|$ everywhere.

  \section{Example}\label{sec:examples}
  Recall that for a warped product of the form $\overline g = dt^2 + \phi(t)^2 ds^2$, the vector field $\phi(t) \partial_t$ is a closed gradient conformal Killing vector field to which Theorem \ref{thm:main-isoperimetric-inequality} applies.
  We include here an example of a conformal Killing vector field which is not closed and thus cannot be so obtained from the warping factor and which is still covered by Theorem \ref{thm:main-isoperimetric-inequality}.

  Consider $\R^3 \setminus \{(0,0,0)\}$ with the conformally flat metric given by
  \[
    g = e^{2f} \overline{g}
  \]
  where $\overline{g}$ is the Euclidean metric and 
  \[
    f = -\ln((x-2)^2+y^2+z^2)
  \]
  We consider $X^\perp = x \frac{\partial}{\partial x} + y \frac{\partial}{\partial y} + z \frac{\partial}{\partial z}$, the position vector field, and $X^\top = -z \frac{\partial}{\partial y} + y \frac{\partial}{\partial z}$, the rotation vector field.
  We let $U = \left\{ x^{2}+y^{2}+z^{2} < 4 \right\} \setminus \left\{ 0 \right\}$ and consider our flow on this open set.
  Now, we compute 
  \begin{align*}
    2\varphi 
    &=
    \mathcal{L}_{X^\perp} g 
    =
    \mathcal{L}_{X^\perp} e^{2f} \overline{g}
    =
    X^\perp(e^{2f}) \overline{g}
    +
    e^{2f} \mathcal{L}_{X^\perp} \overline{g}
    =
    2 e^{2f} X^\perp(f) \overline{g} + 2 e^{2f} \overline{g}
    \\ &=
    2 X^\perp(f) g + 2 g  
    = 
    (2 X^\perp(f) + 2) g  .
  \end{align*}
  We then have $\varphi = (X^\perp(f) + 1)$ and may compute
  \begin{align*}
    1 + X^\perp(f)
    &=
    1 - 
    \frac
    {X^\perp((x-2)^2+y^2+z^2)}
    {(x-2)^2+y^2+z^2}
    =
    1 - 
    \frac
    {2x(x-2)+2y^2+2z^2}
    {(x-2)^2+y^2+z^2}
    \\ &=
    1 - 
    \frac
    {2x(x-2)+2y^2+2z^2}
    {(x-2)^2+y^2+z^2}
    =
    \frac{4-x^2-y^2-z^2}{(x-2)^2+y^2+z^2} 
    = (4-x^2-y^2-z^2) e^f
    \\ &= (4-r^2)e^f,
  \end{align*} 
  which is positive as long as $r^2 = x^2+y^2+z^2 < 4$. 

  Next note that since this is a conformal change of metric it does not change the distribution orthogonal to $X^\perp$ and thus the foliation consists of spheres.
  Then we have 
  \[
    \frac{|X^\perp|^2}{\varphi^2}
    =
    r^2 e^{2f} \frac{1}{(4-r^2)^2 e^{2f}}
    = \frac{r^2}{(4-r^2)^2},
  \]
  which is indeed constant along any sphere centered at the origin and of radius $<2$.
  We also have 
  \begin{align*}
    X^\perp\left( \frac{|X^\perp|^2}{\varphi^2} \right)
    &=
    X^\perp\left( \frac{r^2}{(4-r^2)^2} \right)
    =
    -\frac{r^2 X^\perp\left( (4-r^2)^2 \right)}{(4-r^2)^4}
    +
    \frac{X^\perp \left( r^2 \right)}{(4-r^2)^2}
    \\ &= 
    -\frac{r^2 (2(4-r^2)(-2r)r)}{(4-r^2)^4}
    +
    \frac{2 r^2}{(4-r^2)^2}
    \\ &= 
    \frac{4r^4}{(4-r^2)^3}
    +
    \frac{2 r^2(4-r^2)}{(4-r^2)^3}
    \\ &= 
    \frac{2r^2(4+r^2) }{(4-r^2)^3},
  \end{align*}
  which is positive for $r<2$.
  Next we note that this manifold is isometric to an open set of Euclidean space under a circle inversion followed by a translation and then followed by another circle inversion. Thus this manifold is flat and in particular is Einstein and thus the Ricci curvature conditions trivially hold.
  Finally to check that the vector field is not closed, we compute 
  \[
    \mathrm{d} \left( (X^\perp)^\flat \right) = 
    \mathrm{d} \left( re^{2f} dr \right)
    = \left(- \frac{\partial (e^{2f})}{\partial \theta}r \right) dr \wedge d\theta
    \neq 0.
  \] 

  \section{Appendix} 
  To show uniform regularity of $\lambda$, we will use the following procedure. First we will rewrite the flow as a function flow over a leaf. Next we will examine the flow in coordinates and show it is uniformly parabolic given our bounds on the leaf-wise gradient $\tilde{\nabla} \lambda$ which we have due to our uniform lower bounds on $u$. Finally we will apply known results in parabolic theory to derive $C^\alpha$ estimates on $\lambda$.
  
  First we rewrite this flow as a function flow over a leaf. Consider the integral curves of $X$ starting at the leaf $S_{\lambda_0}$ with $\lambda_0 \leq \min_{\Sigma_0} \lambda$. Since $\Lambda > 0$ and the set of points with $\lambda_0 \leq \lambda \leq \lambda_1$ is compact for any $\lambda_1$ then we have $\Lambda > \varepsilon > 0$ along any such set for some $\varepsilon$. Then along the integral curves of $X$ we have that $\lambda$ has derivative at least $\varepsilon$, we must then have $\lambda$ increase along the curve until in finite time it intersects $S_{\lambda_1}$. Thus by picking a point $p \in S_{\lambda_0}$ picking a value for $\lambda$ satisfying $\lambda_0 \leq \lambda \leq \lambda_1$ we get that the pair $(p, \lambda)$ form a coordinate chart for the region.

  Now letting $\tilde{g}$ be the metric on $S_{\lambda_0}$, we know that $X$ is a conformal killing field with factor $\varphi$ and so we have  
  \[
    \Phi^* \overline{g}(p,\lambda) 
    = \overline{g}(p,\lambda_0) \int_{0}^{t(p,\lambda)} 2\varphi ds
  \]
  where $\Phi^*$ is the flow of $X$ and $t(p,\lambda)$ is the time along the integral curve when we reach $(p,\lambda)$ from $(p,\lambda_0)$.
  Set $G(p,\lambda) = \int_{0}^{t(p,\lambda)} 2\varphi ds$ in this chart.

  Let $F_{t}$ denote the family of embeddings which solves the flow equation.
  Represent $F_t$ for some $t$ as $(p(s), \lambda(s))$, from which we can decompose 
  \[
    \nu = Y + z \frac{\partial}{\partial \lambda}
  \]
  where $z > \varepsilon > 0$ uniformly for some fixed $\varepsilon$ since we have lower bounds on $u = (X,\nu)$. We then must have for any chart $(s^1,\dots,s^n)$ on $\Sigma$ that $\partial_{s^i} (p(s), \lambda(s))$ is non-zero since $F_t$ is an embedding. Thus if $\partial_{s^i} p(s) = 0$ then we have $\partial_{s^i} \lambda(s) \neq 0$ and so $\left((\partial_{s^i} (p(s), \lambda(s))), \nu\right) \neq 0$ which is a contradiction. Thus by reparametrizing we can get $F_t = (p, \lambda(p^{-1}(p)))$ globally. Then by fixing any normal coordinates $p^1,\dots,p^n$ on $S_{\lambda_0}$ we get coordinates on $F_t(\Sigma)$. In these coordinates we have the following expression for the induced metric
  \[
    g_{ij} = G(p,\lambda) \tilde{g}_{ij} + H(p, \lambda) (d\lambda \otimes d\lambda)_{ij}
  \]
  where $H(p,\lambda) = \overline{g}\left( \frac{\partial}{\partial \lambda}, \frac{\partial}{\partial \lambda} \right)$.

  This then gives the following expression for the determinant and inverse metric 
  \begin{gather*}
    \det g = G^n(p, \lambda) \left( 1 + \frac{H(p, \lambda)}{G(p, \lambda)} |\tilde{\nabla} \lambda|^2 \right) \det{\tilde{g}_{ij}},
    \\
    g^{ij} 
    = \frac{1}{G(p, \lambda)} \left( \tilde{g}^{ij} 
    - \frac{H(p, \lambda)}{G(p, \lambda)}\frac{(\tilde{\nabla} \lambda \otimes \tilde{\nabla} \lambda)^{ij}}{1 + \frac{H(p, \lambda)}{G(p, \lambda)}|\tilde{\nabla} \lambda|^2} \right).
  \end{gather*}

  We now have that 
  \[
    \nu = \frac{1}{\sqrt{\left( G(p,\lambda) + H(p,\lambda) |\tilde{\nabla} \lambda|^2 \right)}}(-\tilde{\nabla} \lambda, 1)
  \]
  and so
  \[
    u = \overline{g}\left( X, \frac{\partial}{\partial \lambda} \right) \frac{1}{\sqrt{\left( G(p,\lambda) + H(p,\lambda) |\tilde{\nabla} \lambda|^2 \right)}}
  \]
  and since we have a lower bound on $u$ this gives us an upper bound $c_1$ on $\tilde{\nabla} \lambda$.

  In these leaf coordinates we then have
  \begin{gather*}
    \Delta_g \lambda = \frac{1}{\sqrt{\det g}}\partial_{p^i}\left( g^{ij} \sqrt{\det g} \partial_{p^j} \lambda \right)
    \\ \partial_t \lambda 
    - u \Delta_g \lambda 
    = - u\frac{2}{\varphi^2 |X|^2} X(\Lambda \varphi^2)(|X|^2 - u^2)
    + 4u\frac{\Lambda }{\varphi}(X(\varphi) - u\nu(\varphi)).
  \end{gather*}
  Now we write 
  \[
    \Delta_g \lambda = \frac{1}{\sqrt{\det{g}}} \partial_{p^i} (A^i(p,\lambda,\tilde{\nabla} \lambda))
  \]
  where 
  \[
    A^i(x,z,p) 
    =
    \sqrt{G^n(x, z) \left( 1 + \frac{H(x, z)}{G(x, z)} |p|^2 \right)}
    \frac{1}{G(x, z)} \left( \tilde{g}^{ij} (x)
    - \frac{H(p, \lambda)}{G(p, \lambda)}\frac{p_i p_j}{1 + \frac{H(x, z)}{G(x, z)}|p|^2} \right)
    \tilde{g}_{jk} p_k.
  \]
  Simplifying we get 
  \[
    A^i(x,z,p) = \sqrt{G^n(x, z) \left( 1 + \frac{H(x, z)}{G(x, z)} |p|^2 \right)}
    \frac{1}{G(x, z)} p_i \left( 1
    - \frac{H(p, \lambda)}{G(p, \lambda)}\frac{|p|^2}{1 + \frac{H(x, z)}{G(x, z)}|p|^2} \right)
  \]
  which then further simplifies into 
  \begin{align*}
    A^i(x,z,p) &= \sqrt{G^n(x, z) \left( 1 + \frac{H(x, z)}{G(x, z)} |p|^2 \right)}
    \frac{1}{G(x, z)} p_i \left(\frac{1}{1 + \frac{H(x, z)}{G(x, z)}|p|^2} \right)
    \\
    &=  \sqrt{\frac{G^{n-2}(x, z)}{1 + \frac{H(x, z)}{G(x, z)}|p|^2} } p_i.
  \end{align*}
  Then 
  \[
    \partial_{p^j} A^i(x,z,p) 
    =
    \delta^i_j \sqrt{\frac{G^{n-2}(x, z)}{1 + \frac{H(x, z)}{G(x, z)}|p|^2} }
    -
    \sqrt{\frac{G^{n-2}(x, z)}{1 + \frac{H(x, z)}{G(x, z)}|p|^2} }\frac{\frac{H(x, z)}{G(x, z)}p_j p_i}{1 + \frac{H(x, z)}{G(x, z)}|p|^2}
  \]
  and so our uniform bounds on $\tilde{\nabla} \lambda$ give us a uniform ellipticity on this matrix, that is $c_2(c_1) I \leq \partial_{p^j} A^i(x,z,p) \leq c_3(c_1) I$.

  We now set
  \[
    B(x, \lambda, \tilde{\nabla} \lambda) 
    =
    -u\frac{2}{\varphi^2 |X|^2} X(\Lambda \varphi^2)(|X|^2 - u^2)
    +4u\frac{\Lambda }{\varphi}(X(\varphi) - u\nu(\varphi))
  \]
  and we can rewrite our evolution equation for $\lambda$ as 
  \begin{align*}
    \partial_t \lambda - \frac{u}{\sqrt{\det g}} (\partial_{p^i} A^i(p, \lambda, \tilde{\nabla}\lambda)) =
    B(x, \lambda, \tilde{\nabla} \lambda) .
  \end{align*}
  We then can rewrite this as 
  \[
    \partial_t \lambda - \frac{u}{\sqrt{\det g}}(\partial_{x^i} A^i(x,z,p) + \partial_{z} A^i(x,z,p) (\tilde{\nabla} \lambda)^i + \partial_{p^j} A^i(x,z,p) \partial_{i} \partial_{j} \lambda) =
    B(x, \lambda, \tilde{\nabla} \lambda)
  \]
  Now $\partial_{x^i} A^i(x,z,p), \partial_{z} A^i(x,z,p)$ both involve only spacial derivatives of ambient functions and thus are bounded in our compact set by some constants $c_4(\lambda_0,\lambda_1),c_5(\lambda_0,\lambda_1)$ respectively.

  We can now pick any open ball $\Omega = B_R(p)$ and $\Sigma_{\lambda_0}$, and a compact ball $\Omega' = \overline{B_{R/2}(p)}$, then by Theorem 1.1 on page 517 of \cite{ladyzhenskaia1968linear} we get that 
  \[
    [\tilde{\nabla} \lambda]_{\alpha;\Omega'} \leq 
    C(n,c_1,\lambda_0,\lambda_1,R)
  \]
  We thus get that all the coefficients of this parabolic equation are in $H^{\alpha, \alpha/2}(\Omega)$, with norms bounded by $C(n,c_1,\lambda_0,\lambda_1,R)$.

  Then by Theorem 10.1 on page 351 of \cite{ladyzhenskaia1968linear} we get 
  \[
    |\lambda|_{\Omega'}^{(2 + \alpha)} \leq C(n,c_1,\lambda_0,\lambda_1,R)
  \]
  then this gives us that 
  \[
    |\tilde{\nabla} \lambda|_{1+\alpha;\Omega'}
    \leq 
    C(n,c_1,\lambda_0,\lambda_1,R)
  \]
  which we now use again to get 
  \[
    |\lambda|_{\Omega'}^{(3 + \alpha)} \leq C(n,c_1,\lambda_0,\lambda_1,R)
  \]
  and we can continue this process to get uniform bounds on all $C^\alpha$ norms of $\lambda$, since we can do this around any point without changing $R$ we get uniform bounds over all of $\Sigma$.

  \bibliographystyle{abbrv}

\begin{thebibliography}{10}

\bibitem{alias_constant_2006}
L.~J. Alías, J.~H.~S. de~Lira, and J.~M. Malacarne.
\newblock {Constant} {Higher}-{Order} {Mean} {Curvature} {Hypersurfaces} {in}
  {Riemannian} {Spaces}.
\newblock {\em Journal of the Institute of Mathematics of Jussieu}, 5(04):527,
  Oct. 2006.

\bibitem{andrzejewski_newton_2010}
K.~Andrzejewski and P.~G. Walczak.
\newblock The {Newton} transformation and new integral formulae for foliated
  manifolds.
\newblock {\em Annals of Global Analysis and Geometry}, 37(2):103--111, Feb.
  2010.

\bibitem{gage_heat_1986}
M.~Gage and R.~S. Hamilton.
\newblock The heat equation shrinking convex plane curves.
\newblock {\em Journal of Differential Geometry}, 23(1):69--96, Jan. 1986.
\newblock Publisher: Lehigh University.

\bibitem{guan_mean_2013}
P.~Guan and J.~Li.
\newblock A mean curvature type flow in space forms.
\newblock {\em Int. Math. Res. Not. IMRN}, (13):4716--4740, 2015.

\bibitem{guan_volume_2018}
P.~Guan, J.~Li, and M.-T. Wang.
\newblock A volume preserving flow and the isoperimetric problem in warped
  product spaces.
\newblock {\em Trans. Amer. Math. Soc.}, 372(4):2777--2798, 2019.

\bibitem{huisken_flow_1984}
G.~Huisken.
\newblock Flow by mean curvature of convex surfaces into spheres.
\newblock {\em Journal of Differential Geometry}, 20(1):237--266, Jan. 1984.
\newblock Publisher: Lehigh University.

\bibitem{jiayu_isoperimetric_2023}
  J.~Li and S.~Pan.
\newblock The isoperimetric problem in the {Riemannian} manifold admitting a
  non-trivial conformal vector field, Mar. 2023.
\newblock arXiv:2303.17887 [math].

\bibitem{kwong_extension_2014}
K.-K. Kwong.
\newblock An extension of {H}siung-{M}inkowski formulas and some applications.
\newblock {\em J. Geom. Anal.}, 26(1):1--23, 2016.

\bibitem{ladyzhenskaia1968linear}
O.~A. Ladyzhenskaia, V.~A. Solonnikov, and N.~N. Ural'tseva.
\newblock {\em Linear and quasi-linear equations of parabolic type}, volume~23.
\newblock American Mathematical Soc., 1968.

\bibitem{schulze_nonlinear_2006}
F.~Schulze.
\newblock Nonlinear evolution by mean curvature and isoperimetric inequalities.
\newblock {\em J. Differential Geom.}, 79(2):197--241, 2008.

\bibitem{MR4429248}
J.~V\'{e}tois.
\newblock Convergence result and blow-up examples for the {G}uan-{L}i mean
  curvature flow on warped product spaces.
\newblock {\em Comm. Anal. Geom.}, 29(8):1917--1935, 2021.

\end{thebibliography}

  \end{document}